\documentclass[11pt]{article}
\usepackage[text={6.5in,8.4in},centering]{geometry}
\usepackage{graphicx,url,subfig}
\RequirePackage[OT1]{fontenc}
\RequirePackage{amsthm,amsmath,amssymb,amscd}
\RequirePackage[numbers]{natbib}
\RequirePackage[colorlinks,citecolor=blue,urlcolor=blue]{hyperref}
\usepackage{enumerate}
\usepackage{datetime}
\usepackage{etex}

\usepackage{mathrsfs}
\usepackage{amsfonts}
\usepackage{latexsym,bm,amsfonts,amssymb,pifont}
\usepackage{amsmath}
\usepackage{color}
\usepackage{empheq}

\usepackage{enumerate}

\makeatother
\numberwithin{equation}{section}
\allowdisplaybreaks

\usepackage{enumerate}

\newtheorem{theorem}{Theorem}[section]

\newtheorem{thm}[theorem]{Theorem}

\newtheorem{lemma}[theorem]{Lemma}
\newtheorem{lem}[theorem]{Lemma}

\numberwithin{equation}{section}
\numberwithin{equation}{section}
\newtheorem{pro}[theorem]{Proposition}

\newtheorem{co}[theorem]{Corollary}
\newtheorem{de}[theorem]{Definition}

\newtheorem{rem}[theorem]{Remark}

\def\thl{\text{-H\"ol}}

\def\de{\delta}

\def\si{{\sigma}}

\def\RR{\mathbb{R}}

\def\EE{\mathbb{E}}

\def\de{{\delta}}

\def\si{{\sigma}}

\def\de{{\delta}}

\def\si{{\sigma}}

\def\vare{{\varepsilon}}

\def\th{{\theta}}

\def \comb#1#2{{\left({#1}\atop{#2}\right)}}

\def\dom{{\rm Dom}}
\def\cov{{\rm Cov}}
\def\sgn{{\rm sgn}}

\begin{document}
	
\date{}
\title{ Parameter estimation for fractional Ornstein-Uhlenbeck processes
of general Hurst parameter}

  \author{  {\sc Yaozhong Hu}\thanks{Y.  Hu is
partially supported by a grant from the Simons Foundation
\#209206.},   \;\;     {\sc David Nualart}\thanks{D. Nualart is supported by the NSF grant  DMS1208625.  
 },\;\; and \;\;  {\sc Hongjuan Zhou} }

 \maketitle
 \begin{abstract}

This paper provides  several  statistical estimators
for the  drift and  volatility parameters of  an Ornstein-Uhlenbeck process 
driven by fractional Brownian motion,  whose observations  
 can be made either continuously or at discrete time instants.  
First   and higher order power variations   are
used  to estimate  the volatility
parameter. The almost sure convergence of the estimators and the  corresponding 
 central limit theorems are obtained  
for all the  Hurst parameter range $H\in (0, 1)$.   
The least squares  estimator  is
used for   the drift parameter.
A    central limit theorem  is proved 
when the Hurst parameter $H \in (0, 1/2)$
and a noncentral limit theorem is proved for $H\in[3/4, 1)$.  Thus,  the open problem left in the previous paper \cite{hn} is completely solved,  where a central limit theorem for least squares estimator is proved for $H\in [1/2, 3/4)$.  
\end{abstract}

\vspace{10pt}

\numberwithin{equation}{section}

\section{Introduction}

Consider the  fractional Ornstein-Uhlenbeck process  
defined as the unique 
pathwise solution to the  stochastic differential equation
  \begin{equation}
  d X_t =   - \theta  X_t dt +   \sigma_t dB_t^H\,,  
  \label{e.ou-time-varying}
  \end{equation}
  with initial condition $X_0\in \mathbb{R}$, where $B^H=\{B^H_t, t\ge 0\}$ is a fractional Brownian  (fBm) motion of  Hurst parameter
  $H\in (0,1)$, $\theta$ is a positive parameter and 
the volatility $\si_t$  is a stochastic process with $\beta$-H\"older continuous trajectories, where $\beta>1-H$.
Under this  condition on $\si_t$, the stochastic integral $\int_0^t \sigma_s dB^H_s$ is well defined as a
pathwise
Riemann-Stieljes integral (see, for instance, \cite{yo}) and the above stochastic differential equation has a unique solution.

Assume that the parameters $\th>0$  and   $\si_t$ are unknown and that the process
 can be observed continuously or at discrete time
instants. We want to estimate the integrated
 volatility $\int_0^t |\si_s|^pds$ and the drift parameter $\theta$ for any $H \in (0,1)$.  We assume that the Hurst parameter $H$ is known
or it can be estimated by other methods (for example, see \cite{bhoz} and the references therein).

In the paper \cite{cnw},    Nualart, Corcuera and Woerner  studied  the asymptotic
behavior of the power variation of
   the stochastic integral $Z_t=\int^t_0 u_s dB^H_s$,  defined as  $V^n_{p}(Z)_t=\sum^{[nt]}_{i=1} |  Z_{i/n}-Z_{(i-1)/n}|^p$ for any $p>0$. They proved that 
   if the process $u=\{u_t, t\ge 0\}$ has finite $q$-variation on any finite interval, for some $q< 1/(1-H)$, then,   as $n\rightarrow \infty$,  
   \[
  n^{-1+pH} V^n_{p}(Z)_t\rightarrow c_{1,p}\int_0^t|u_s|^pds 
  \]
  uniformly in probability in any compact sets of $t$, where $c_{1,p}=\mathbb{E}[|B^H_1|^p]$.
    The corresponding  central limit theorem was also obtained
 for  $H \in (0,\frac{3}{4}]$.
 These  results can be applied to construct an estimator {
 based on the power variation of $\int_0^t \sigma_s dB^H_s$ to estimate} the integrated
 volatility $\int_0^t |\si_s|^pds$  
 when $H \in (0,\frac{3}{4}]$.  However, the condition $H \in (0,\frac{3}{4}]$ is  critical in \cite{cnw}.  The first objective of this paper is to remove this restriction. To this end,
  we shall use higher order power variations defined as
 $V^n_{k,p}(Z)_t = \sum^{[nt]-k+1}_{i=1}
 \left |    \sum_{j=0}^k
 (-1)^{k-j}\comb{k}{j} Z_{(i+j-1)/n} \right |^p$, for any integer $k\ge 1$.
 In Section
 \ref{s.power-variation}, we study the asymptotic behavior of these higher order power variations of the general stochastic integral $Z_t = \int_0^t u_s dB_s^H$.   The application of these results  to estimate
  the integrated volatility are presented in Section \ref{s.volatility}.      In particular,
when $\si_t=\si$ we can use $\displaystyle
\left|\hat\sigma_T\right|^p = \displaystyle\frac{n^{-1+pH}V^n_{k,p}(X)_T}{c_{k,p} T}$
to estimate $\si$,  where $c_{k,p}$ is the constant introduced in (\ref{ckp}). The uniform convergence in probability and central limit theorems
of the estimators for both the integrated volatility  and  the volatility itself are established.

It is worth mentioning that the statistical estimation of  the integrated volatility has already
been studied in the recent decades. Barndorff-Nielsen {\it et} {\it al} \cite{bcp} - \cite{bgj}) studied estimation of  volatility for Brownian semimartingale and
Brownian semi-stationary processes by using  power, bipower, or multipower variations.  
However, those results cannot be applied to the fractional Ornstein-Uhlenbeck process due to its  lack of  the semimartingale property.

As for  the drift parameter $\theta$, several estimators have been proposed previously.  A summary of 
some relevant results are presented below.
\begin{description}
\item[(i)]\  In the case of continuous observations,
Kleptsyna and Le Breton (\cite{kl}) studied  the maximum likelihood estimator (MLE) which is defined by 
\[
 \hat\theta_{MLE} = - \left\{ \int_0^T Q^2(s) dw_s^H\right\}^{-1} \int_0^T Q(s) dZ_s \,,
\]
where
\[
Q(t) = \frac{d}{dw_t^H} \int_0^t k_H(t,s)X_s ds, \quad Z_t = \int_0^t k_H(t,s) dX_s, 
\]
$k_H(t,s) = \kappa_H^{-1} s^{\frac{1}{2}-H} (t-s)^{\frac{1}{2} - H}$ 
and  $w_t^H = \lambda_H^{-1} t^{2-2H}$ with constants $\kappa_H$ and $\lambda_H$   depending on $H$.
They proved  the  almost sure convergence of $\hat\theta_{MLE}$ to $\theta$ as $T$ tends to infinity.  It is worth noting that Tudor and Viens (\cite{tv}) have also obtained the almost sure convergence of both the MLE and  a version of the MLE using discrete observations  for all $H \in (0,1)$.  
Bercu, Courtin and Savy proved in \cite{bcs} the  following   central limit theorem for the MLE in the case of $H>\frac{1}{2}$:
\[
  \sqrt{T} (\hat\theta_{MLE} - \theta) \xrightarrow[T \to \infty]{\mathcal{L}} N(0, 2\theta) \,.
\]
They claimed without proof that the above convergence is also valid for $H \in (0,\frac{1}{2})$. 

\item[(ii)]  On the other hand, Hu and Nualart (\cite{hn}) proposed the least square estimator defined by
\begin{equation}
\hat\theta_T =-\displaystyle\frac{\int^T_0 X_t dX_t}{\int^T_0 X_t^2 dt}=\theta-\sigma \frac{\int^T_0 X_t dB_t^H}{\int^T_0 X_t^2 dt}\,, \label{e.1.2}
\end{equation}
where  the integral with respect to $B^H$  is interpreted in the Skorohod sense.
 They also introduced another estimator  $\widetilde\theta_T$ based on
the ergodic theorem  given by
\begin{equation}
\widetilde{\theta}_T = \Big ( \displaystyle\frac{1}{\sigma^2 H\Gamma(2H)T} \int^T_0 X^2_t dt \Big)^{-\frac{1}{2H}}\,.
\end{equation}
Almost sure convergence and central limit theorems for these two estimators have been proved for $H\in [\frac{1}{2}, \frac{3}{4})$.
\end {description}

However, when  $H\in (0, \frac{1}{2}) \cup [\frac{3}{4}, 1)$, the central limit theorems for  the least square estimator  $\hat\theta_T$ have not been known  yet. The first objective of  Section \ref{s.drift} is to prove the asymptotic consistency of $\hat\theta_T$  by using a new method,  different from that in \cite{hn},  which is valid for all $H \in (0,1)$. This method involves the relationship between the divergence  and Stratonovich integrals  and the integration by parts technique and it is based on the pathwise properties of the fractional Ornstein-Uhlenbeck process established  in a  paper \cite{ckm} by Cheridito, Kawaguchi and Maejima. The next and the main objective of this paper is 
to establish a central limit theorem for the least
square estimator $\hat\theta_T$ for $H\in (0, \frac{1}{2}) $
and  a noncentral limit theorem for $H\in [\frac{3}{4}, 1)$.  In the later case, we can identify
the limit as a Rosenblatt random variable. We will make a comparison of the asymptotic variance for these three estimators and show that the least square estimator performs better than the maximum likelihood estimator when $H \in (0,\frac{1}{2})$. Since the ergodic-type estimator $\widetilde\theta_T$ is a function of a pathwise Riemann integral that appears simpler than the other two estimators, we will use  $\widetilde\theta_T$ to construct a consistent estimator $\bar\theta_n$ for  high frequency data (if only discrete observations are available). The asymptotic behavior of $\bar\theta_n$ in this case is also studied in this paper.
The proofs of our results are highly  technical and rely on some sophisticated computation, which we shall put in the Appendix.
The main tool we use is Malliavin calculus which is recalled
in Section \ref{s.preliminary}.  We use $C$ to denote a generic constant that may vary according to the context.

\section{Preliminaries}\label{s.preliminary}

In this section, we briefly recall some notions and results
on fractional Brownian motion, $p$-variation, and Malliavin  calculus.

The fractional Brownian motion (fBm) $B^H=\{ B^H_t, t \in \mathbb{R} \}$ with Hurst parameter $H \in (0,1)$ is a zero mean Gaussian process, defined on a complete probability space $(\Omega, \mathcal{F}, P)$, with the following covariance function
\begin{equation}
 \mathbb{E}(B^H_tB^H_s)= R_H(t,s)=\frac{1}{2}(|t|^{2H}+|s|^{2H}-|t-s|^{2H}).
 \label{1eq1}
\end{equation}
From \eqref{1eq1}, it is easy to see that $\mathbb{E}|B_t^H-B_s^H|^2 = |t-s|^{2H}$.  Then it follows from Kolmogorov's continuity criterion that on any finite interval, almost surely all paths 
of fBm are $\alpha$-H\"older continuous with $\alpha < H$.  Denote by $\eta_T$  the $\alpha$-H\"older coefficient of fBm on the interval $[0,T]$, i.e., 
\begin{equation}\label{hd.fbm}
  \eta_T = \sup_{t \neq s \in [0,T]} \frac{\left|B_t^H - B_s^H\right|}{|t-s|^{\alpha}} \,.
\end{equation}
Clearly, $\mathbb{E} |\eta_T|^q = T^{q(H-\alpha)} \mathbb{E} |\eta_1|^q$ for any $q>1$, by the self-similarity property of fBm.\\

Let   $\mathcal{F}$ denote the $\sigma$-field obtained from  the completion of the $\sigma$-field generated by   $B^H$.  Let $\mathcal{E}$
denote the space of all real valued step functions on $\mathbb{R}$. The Hilbert space $\mathfrak{H}$ is defined as the closure of $\mathcal{E}$ endowed with the inner product
\[
\langle \mathbf{1}_{[a,b]}, \mathbf{1}_{[c, d]}\rangle_{\mathfrak{H}}= \mathbb{E} \left((B_b^H - B_a^H)(B_d^H - B_c^H)\right) \,.
\]
Under the convention that $\mathbf{1}_{[0,t]} = - \mathbf{1}_{[t,0]}$ if $t<0$, the mapping $\mathbf{1}_{[0,t]} \mapsto B^H_t$ can be extended to a linear isometry between $\mathfrak{H}$ and the Gaussian space $\mathcal{H}_1$ spanned by $B^H$. We denote this isometry by $\mathfrak{H} \ni  \varphi \mapsto B^H(\varphi)$.
If $f, g \in \mathfrak{H}$ and $g$ is  a continuously differentiable function with compact support, we can use step functions in $\mathcal{E}$ to approximate $f$ and $g$ and  by a limiting  argument we deduce
\begin{equation} \label{hinnerp}
   \langle f, g \rangle_{\mathfrak{H}} = \int_{\mathbb{R}^2} f(t)g'(s)\frac{\partial R_H(t,s)}{\partial t} dtds
\end{equation}
(see \cite{hjt}). We can also use Fourier transform to compute $ \langle f, g \rangle_{\mathfrak{H}}$, namely,
\begin{equation}
 \langle f, g \rangle_{\mathfrak{H}}= \frac{1}{c_H^2} \int_{\mathbb{R}} \mathcal{F} f(\xi) \overline{\mathcal{F} g(\xi)} |\xi|^{1-2H} d \xi, \label{sfbm}
\end{equation}
where $c_H = \Big(\frac{2\pi}{\Gamma(2H+1)\sin(\pi H)}\Big)^{\frac{1}{2}}$ (see \cite{pta}). When $H>1/2$, for any $f,g \in L^{1/H}([0,T])$, if we extend $f$ and $g$ to be zero on $\mathbb{R} \cap [0,T]^c$, then $f,g\in \mathfrak{H}$ and we have the following simple identity
\begin{equation} \label{ip.ghalf}
	\langle f, g \rangle_{\mathfrak{H}} = \alpha_H \int_{[0,T]^2} f(u) g(v) |u-v|^{2H-2} du dv \,,
\end{equation}
where $\alpha_H = H(2H-1)$.

For any $p>0$,  the $p$-variation of a real-valued function $f$ on an interval $[a,b]$ is defined as
\[
{\rm var}_p(f;[a,b])={\rm sup}_{\pi}\Big(\sum_{i=1}^n|f(t_i)-f(t_{i-1})|^p\Big)^{1/p}\,,
\]
where the supremum runs over all partitions $\pi=\{a=t_0<t_1<\cdots<t_n=b\}$.
If $f$ is $\alpha$-H\"{o}lder continuous on the interval $[a, b]$,  $\alpha \in (0,1]$, then we set
\[
\|f\|_{\alpha}:= {\rm sup}_{a\leq s<t \leq b}\frac{|f(t)-f(s)|}{|t-s|^{\alpha}}\,.
\]
It is known that an $\alpha$ \thl der continuous function $f$ on the interval $[a, b]$ has finite $1/\alpha$-variation on this  interval.
If $f$ and $g$ have finite $p$-variation and finite $q$-variation on the interval $[a,b]$ respectively and $1/p+1/q>1$, the Riemann-Stieltjes integral $\int_a^b fdg$ exists (see Young \cite{yo}).  
By Young's result, the stochastic integral $\int_0^t u_s dB_s^H$ is well defined as a pathwise Riemann-Stieltjes integral provided that the trajectories of the process $\{u_t, t\geq 0\}$ have finite $q$-variation on any finite interval for some $q<1/(1-H).$ 

Next we define two types of stochastic integrals: Stratonovich integral and divergence integral. Given a stochastic process $\{v(t), t\geq 0\}$ such that $\int_0^t |v(s)|ds < \infty$ a.s. for all $t>0$, the Stratonovich integral $\int_0^t v(s) \circ dB_s^H$
 is defined as the following limit in probability if it exists
 \[
   \lim_{\epsilon \to 0}\int_0^t v(s) \dot{B}_s^{H,\vare}  ds \,,
 \]
 where $\dot{B}_s^{H,\vare}$ is a symmetric approximation of $\dot{B}_s^{H }$: 
 \[
 \dot{B}_s^{H,\vare}= \frac{1}{2\epsilon} (B_{s+\epsilon}^H - B_{s-\epsilon}^H) \,.
 \]

Before we  define the divergence integral, we present  some background of Malliavin calculus. For a smooth and cylindrical random variable $F= f(B^H(\varphi_1), \dots , B^H(\varphi_n))$, with $\varphi_i \in \mathfrak{H}$ and $f \in C_b^{\infty}(\mathbb{R}^n)$ ($f$ and all of its partial derivatives are bounded), we define its Malliavin derivative as the $\mathfrak{H}$-valued random variable given by
\[
 DF = \sum_{i=1}^n \frac{\partial f}{\partial x_i} (B^H(\varphi_1), \dots, B^H(\varphi_n))\varphi_i.
 \]
By iteration, one can define the $k$-th derivative $D^k F$  as an element of $L^2(\Omega; \mathfrak{H}^{\otimes k})$. For any natural number $k$ and any real number $ p \geq 1$, we define  the Sobolev space $\mathbb{D}^{k,p}$  as the closure of the space of smooth and cylindrical random variables with respect to the norm $||\cdot||_{k,p}$ defined by 
\[
||F||^p_{k,p} = \mathbb{E}(|F|^p) + \sum_{i=1}^k \mathbb{E}(||D^i F||^p_{\mathfrak{H}^{\otimes i}}).
\]

The divergence operator $\delta$ is defined as the adjoint of the derivative operator $D$ in the following manner.  An element $u \in L^2(\Omega; \mathfrak{H})$ belongs to the domain of $\delta$, denoted 
by $\dom\, \delta$, if there is a constant $c_u$ depending on $u$ such that 
\[
|\mathbb{E} (\langle DF, u \rangle_{\mathfrak{H}})| \leq c_u ||F||_{L^2(\Omega)}
\] for any $F \in \mathbb{D}^{1,2}$.  If $u \in \dom \,\delta$, then the random variable $\delta(u)$ is defined by the duality relationship 
\[
\mathbb{E}(F\delta(u)) = \mathbb{E} (\langle DF, u \rangle_{\mathfrak{H}}) \, ,
\] 
which holds for any $F \in \mathbb{D}^{1,2}$.  
If $u=\{u_t, t\in [0,T]\}$ is a stochastic process, whose trajectories belong to $\mathfrak{H}$ almost surely (with the convention $u_t=0$ if  $t\not\in [0,T]$) and $u\in \dom \, \delta$, we make use of the notation $ \int_0^T u_t dB_t^H=\delta(u)$ and call $\delta(u)$ the   divergence integral of $u$ with respect to the fractional Brownian motion $B^H$ on $[0,T]$.
It is worth noting that the  divergence integral of fBm with respect to itself does not exist if $H \in (0,\frac{1}{4})$ because the paths of  the fBm are too irregular (see \cite{cn}). For this reason, in \cite{cn} the authors introduce an extended divergence integral $\delta^*$ such that $\dom \,\delta^* \cap  L^2 (\Omega; \mathfrak{H}) = \dom\, \delta$ and the extended divergence operator $\delta^*$ restricted to $\dom\, \delta$  coincides with the  divergence operator.
In a similar way we can introduce the iterated divergence operator $\delta^k$ for each integer $k\ge 2$, defined by the duality relationship 
\[
\mathbb{E}(F\delta^k(u)) = \mathbb{E}  \left(\langle D^kF, u \rangle_{\mathfrak{H}^{\otimes k}} \right) \, ,
\] 
 for any $F \in \mathbb{D}^{k,2}$, where $u\in \dom \, \delta^k \subset L^2(\Omega; \mathfrak{H}^{\otimes k})$.

For any integer $m \geq 1$, we use $\mathfrak{H}^{\otimes m}$ and $\mathfrak{H}^{\odot m}$ to denote the $m$-th tensor product and the 
$m$-th symmetric tensor product of the Hilbert space $\mathfrak{H}$, respectively. We denote by $\mathcal{H}_m$ the closed linear subspace of $L^2(\Omega)$ generated by the random variables $\{H_m(B^H(\varphi)): \varphi \in \mathfrak{H}, ||\varphi||_{\mathfrak{H}}=1\}$, where $H_m$ is the $m$-th Hermite polynomial defined by 
\[
H_m(x)=\frac{(-1)^m}{m!}e^{\frac{x^2}{2}}\frac{d^m}{dx^m}e^{-\frac{x^2}{2}},\quad m \geq 1,
\]
and $H_0(x)=1$. The space $\mathcal{H}_m$ is called the Wiener chaos of order $m$. The $m$-th multiple integral of $\varphi \in \mathfrak{H}^{\odot m}$ is defined by the identity $I_m(\varphi) = \delta^m (\varphi)$, and in particular,  $ I_m(\phi^{\otimes m}) = H_m(B^H(\phi))$ for any $\phi\in \mathfrak{H}$. The map $I_m$ provides a linear isometry between $\mathfrak{H}^{\odot m}$(equipped with the norm $\frac{1}{\sqrt{m!}}||\cdot||_{\mathfrak{H}^{\otimes m}}$) and $\mathcal{H}_m$ (equipped with $L^2(\Omega)$ norm) (see \cite{np}, Theorem 2.7.7). By convention, $\mathcal{H}_0 = \mathbb{R}$ and $I_0(x)=x$.

Let us recall the definition of the Rosenblatt process that will  appear in the the limit theorems of Section \ref{s.drift}.    Fix $H>3/4$ and $t \in [0,1]$. Consider the sequence  of functions of two variables 
\[
\xi_{n,t} = 2^n\sum\limits_{i=1}^{[2^n t]} {\bf 1}_{((i-1)2^{-n}, i2^{-n}]}^{\otimes 2} \,.
\]
Through a direct computation using \eqref{ip.ghalf} one can show that this sequence is   Cauchy in $\mathfrak{H}^{\otimes{2}}$ and converges to  distribution denoted by  $\delta_{0,t}$ and defined by
\begin{equation} \label{ddel.def}
 \langle \delta_{0,t} , f\rangle = \int_0^t f(s,s)ds,  
\end{equation} 
for any test function $f$ on $\mathbb{R}^2$. It turns out (see   \cite{nnt} for the proofs)  that    the sequence $I_2(\xi_{n,t})$ converges in $L^2$ as $n$ tends to infinity to the Rosenblatt random variable $R_t = I_2(\delta_{0,t})$.   For any $f \in L^{1/H} ([0,1]^2)$,  we have the following formula, letting $f$ equal to zero on $\mathbb{R}^2 \cap [0,1]^c$, 
\begin{equation} \label{ip.rosenb}
	 \mathbb{E} (R_t I_2(f)) = 2 \langle \delta_{0,t}, f  \rangle_{\mathfrak{H}^{\otimes 2}} = 2\alpha_H^2 \int_0^t dv \int_{[0,1]^2} f(u_1, u_2) |u_1 - v|^{2H-2} |u_2 - v|^{2H-2} du_1 du_2 \,.
\end{equation}

The space $L^2(\Omega)$ can be decomposed into the infinite orthogonal sum of the spaces $\mathcal{H}_m$, which is known as the Wiener chaos expansion. Thus, any square integrable random variable $F \in L^2(\Omega)$ has the following expansion,
\[
  F = \sum_{m=0}^{\infty} I_m (f_m) ,
\]
where $f_0 = \mathbb{E}(F)$, and $f_m \in \mathfrak{H}^{\odot m}$ are uniquely determined by $F$. We denote by $J_m$ the orthogonal projection onto  the $m$-th Wiener chaos $\mathcal{H}_m$. This means 
that  $I_m (f_m) = J_m (F)$ for every $m \geq 0$.

Let $\{e_k, k \geq 1\}$ be a complete orthonormal system in the Hilbert space $\mathfrak{H}$. Given $f \in \mathfrak{H}^{\odot n}, g \in \mathfrak{H}^{\odot m}$, and $p=1,\dots,n \wedge m$, the $p$-th contraction between $f$ and $g$ is the element of $\mathfrak{H}^{\otimes (m+n-2p)}$ defined by
\begin{equation*}
 f \otimes_p g = \sum_{i_1,\dots,i_p=1}^{\infty} \langle f,e_{i_1} \otimes \cdots \otimes e_{i_p} \rangle_{\mathfrak{H}^{\otimes p}} \otimes \langle g,e_{i_1} \otimes  \cdots  \otimes e_{i_p} \rangle_{\mathfrak{H}^{\otimes p}} \,.
\end{equation*}

The following result (known as the fourth moment theorem)  provides necessary and sufficient conditions for the  convergence of some  random variables to a normal distribution (see \cite{no,dngp,np}).
\begin{thm} \label{fm.theorem}
 Let $ n \geq 2 $ be a fixed integer. Consider  a collection of elements $\{f_{T}, T>0\}$ such that $f_{T} \in \mathfrak{H}^{\odot n}$ for every $T>0$. Assume further that
 \[
  \lim_{T \to \infty}\mathbb{E} [I_n (f_T) ^2] = \lim_{T \to \infty} n! \|f_T\|^2_{\mathfrak{H}^{\otimes{n}}} = \sigma^2 .
 \]
Then the following conditions are equivalent:
 \begin{enumerate}
	\item $\lim_{T \to \infty} \mathbb{E}[I_n(f_T)^4] = 3\sigma^2$.
	\item For every $p=1,\dots, n-1$, $\lim_{T \to \infty}||f_T \otimes_p f_{T}||_{\mathfrak{H}^{\otimes 2(n-p)}} = 0$.
    \item As $T $ tends to infinity, the $n$-th multiple integrals $\{I_n (f_T), T \geq 0\}$ converge in distribution to a standard Gaussian random variable $N(0,\sigma^2)$.
    \item $\|D(I_n(f_T))\|_{\mathfrak{H}}^2 \xrightarrow[T \to \infty]{L^2(\Omega)} n \sigma^2$.
 \end{enumerate}	
\end{thm}

\begin{rem}
 The multidimensional version of the above theorem is also stated and proved in \cite{no,np,pt}. 
\end{rem}
In the paper \cite{no}, Nualart and Ortiz-Lattore apply the fourth moment theorem to establish the following weak convergence result for an arbitrary sequence of centered square integrable random vectors.
\begin{thm}\label{l2.clt}
	Let $\{F_k, k \in \mathbb{N}\}$ be a sequence of $d$-dimensional centered square integrable random vectors with the following Wiener chaos expansions:
	\[
	F_k = \sum_{m=1}^{\infty} J_m F_k \,.
	\]
	Suppose that:
	\begin{itemize}
	  \item[(i)] $\lim_{M \to \infty} \limsup_{k \to \infty} \sum_{m=M+1}^{\infty} \mathbb{E}[|J_mF_k|^2] = 0$ .
	  \item[(ii)] For every $m \geq 1$, $1 \leq i, j \leq d$, $\lim_{k \to \infty} \mathbb{E}[(J_mF_k^i)(J_m F_k^j)] = C_m^{ij}$.
	  \item[(iii)] For all $v \in \mathbb{R}^d$, $\sum_{m=1}^{\infty} v^T C_m v = v^T C v$, where $C$ is a $d \times d$ symmetric nonnegative definite matrix.
	  \item[(iv)] For all $m \geq 1$, $1 \leq i, j \leq d$, $$\langle D(J_m F_k^i), D(J_m F_k^j)\rangle_{\mathfrak{H}} \xrightarrow[k \to \infty]{L^2(\Omega)} mC_m^{ij} \,.$$
	\end{itemize}
    Then, $F_k$ converges in distribution to the $d$-dimensional normal law $N_d(0,C)$ as $k$ tends to infinity.
\end{thm}

We end this section by stating the following theorem proved in the paper \cite{cnp} on the asymptotic behavior of weighted random sums.  It will be used in the next section to prove the central limit theorem of the power variation of stochastic integrals.
\begin{thm} \label{wrsum}
Let $(\Omega, \mathcal{F}, P)$ be a complete probability space. Fix a time interval $[0,T]$ and consider a double sequence of random variables $\xi = \{\xi_{i,m}, m \in \mathbb{Z}_+, 1 \leq i \leq [mT] \}$. Assume the double sequence $\xi$ satisfies the following hypotheses. \\
 (H1) Denote $g_m(t):= \sum_{i=1}^{[mt]} \xi_{i,m}$. The  finite dimensional distributions of the sequence of processes $\{g_m(t), t \in [0,T]\}$ converges $\mathcal{F}$-stably  to those of $\{B(t), t \in [0,T]\}$  as $m \to \infty$, where $\{B(t), t \in [0,T]\}$ is a standard Brownian motion independent of $\mathcal{F}$.\\
 (H2) $\xi$ satisfies the tightness condition $\mathbb{E} \left|\sum_{i=j+1}^k \xi_{i,m}\right|^4 \leq C \left(\frac{k-j}{m}\right)^2$ for any $1 \leq j < k \leq [mT]$.\\
If $\{f(t), t \in [0,T]\}$ is an $\alpha-$H\"older continuous process with $\alpha > 1/2$ and we set $X_m(t) := \sum_{i=1}^{[mt]} f(\frac{i}{m})\xi_{i,m}$, then we have the  $\mathcal{F}$-stable convergence
\[
 X_m(t) \xrightarrow[m \to \infty]{\mathcal{L}} \int_0^t f(s) dB(s), 
\]
in the Skorohod space $\mathcal{D}[0,T]$.
\end{thm}

\section{Asymptotic behavior of power variation}\label{s.power-variation}

In this section, we introduce high order power variations and prove
some asymptotic results for the high order power variations of stochastic integrals
with respect to fBm. The high order power variations  will be used to construct estimators for
the volatility and the integrated volatility of fractional
Ornstein-Uhlenbeck processes in the next section.

Consider a sequence of random variables $\{X_{i-1} \,, i \geq 1\}$.  Denote the first order difference 
 $\Delta X_{i-1}=\Delta_1 X_{i-1}=X_{i}-X_{i-1}$.  Define the $k$-th order difference by induction as follows
  $\Delta_k X_{i-1} = \Delta_{k-1} X_{i}-\Delta_{k-1} X_{i-1}$ for  $k=2, 3, \dots$,  namely,
\[
\Delta_k X_{i-1} =  \sum_{j=0}^k
 (-1)^{k-j}\comb{k}{j} X_{i+j-1}\,.
\]
Let $B^H = \{B^H_t, t\geq 0\}$ be a  fBm with Hurst parameter $H \in (0,1)$. For any $j\ge 0$, we can write down the covariance  function  of the $k$-th  order difference of the sequences
$\{B^H_n, n\ge 0\}$ and  $\{B^H_{n+j}, n\ge 0\}$ as   follows
\[
 \rho_{k,H}(j):=\mathbb{E}[(\Delta_kB^H_{n+j})(\Delta_kB^H_{n})]=\frac{1}{2} \sum_{i=-k}^{k}(-1)^{1-i}\dbinom{2k}{k-i}|j-i|^{2H}\,.
 \label{rkj}
\]
Since all the moments of a mean zero Gaussian can be expressed by its variance, we see that the $p$-th moment   of 
  $\Delta_kB^H_{n}$    is given by
\begin{equation}\label{ckp}
  c_{k,p} = \mathbb{E}[|\Delta_k B_n^H|^p] = \displaystyle\frac{2^{p/2} \Gamma((p+1)/2)}{\Gamma(1/2)}[\rho_{k,H}(0)]^{p/2}.
\end{equation}
Notice that the quantities $\rho_{k,H}(j)$ and $c_{k,p}$ are independent of $n$, due to the fact that the  fBm has stationary increments.

From the fact that  $\rho_{k,H}(j)= o(j^{2H-2k}) $ for $j$ large it follows that
\[
  \sum_{j=0}^{\infty}\rho_{k,H}^2(j)
\begin{cases}
= \infty&\qquad \hbox{when $k=1$ and  $\frac{3}{4} \leq H < 1$},\\
< \infty &\qquad \hbox{when $k=1$ and   when $0 < H < \frac{3}{4}$},\\
< \infty &\qquad \hbox{when $k\geq 2$}\,.
 \end{cases}
\]

Let  $p>0$ and let  $n \geq 1$ be an integer.   We define the
$k$-th order $p$-variation  of a stochastic process $Z=\{Z_t,t\geq 0\}$  as
\begin{equation}\label{vkp}
	V^n_{k,p}(Z)_t = \sum^{[nt]-k+1}_{i=1} |\Delta_k Z_{\frac{i-1}{n}}|^p
	=\sum^{[nt]-k+1}_{i=1}
 \left |    \sum_{j=0}^k
 (-1)^{k-j}\comb{k}{j} Z_{\frac{i+j-1}{n}} \right |^p\,,
\end{equation}
where we use the convention  that the
sum is   zero if $[nt]-k+1<1$.

The following proposition shows the convergence of the $k$-th order $p$-variation
for stochastic integrals of fractional Brownian motion, extending a result in   \cite{cnw}  which is valid when  $k=1$.

\begin{theorem} \label{pro1}
Let $k \geq 2$  and let $H\in(0,1)$.  Suppose that $\{u_t, t \in [0,T]\}$ is a stochastic process whose sample paths are  H\"older continuous with exponent $1/q$ for a certain
  $q < \frac 1 {1-H}$.  Consider the pathwise Riemann-Stieltjes integral
\[
Z_t=\int^t_0 u_s dB^H_s\,, \quad t\in [0,T].
\]
Then for any $p> 0$,  as $n\rightarrow \infty$,
\begin{equation}
n^{-1+pH}V^n_{k,p}(Z)_t \rightarrow  c_{k,p} \int^t_0 |u_s|^pds \label{e.limit-v_n}
\end{equation}
 in probability, uniformly on $[0,T]$,  where $c_{k,p}$ is the constant introduced in (\ref{ckp}).
\end{theorem}

\begin {proof}
Denote by $\|\cdot \|_\infty$  the supremum norm on $[0, T]$. For any $t \in [0,T]$ and  any $m\geq n \geq 1$, by the definition of $V^m_{k,p}(Z)_t$, we have   
\begin {eqnarray*}
& & m^{-1+pH}V^m_{k,p}(Z)_t - c_{k,p} \int^t_0|u_s|^p ds \\
& = & m^{-1+pH} \sum^{[mt]-k+1}_{i=1} \left ( \left | \Delta_kZ_{\frac{i-1}{m}} \right |^p - \left |u_{\frac{i}{m}}\Delta_k B^H_{\frac{i-1}{m}} \right |^p \right ) \\
& & + \; m^{-1+pH} \left( \sum^{[mt]-k+1}_{i=1} \left | u_{\frac{i}{m}}\Delta_k B^H_{\frac{i-1}{m}} \right |^p - \sum^{[nt]-k+1}_{i=1} \left | u_\frac{i-1}{n} \right |^p\sum_{j\in I_n(i)} \left |\Delta_k B^H_{\frac{j-1}{m}} \right |^p \right) \\
& & + \; m^{-1+pH} \sum^{[nt]-k+1}_{i=1} \left|u_\frac{i-1}{n}\right|^p\sum_{j\in I_n(i)} \left|\Delta_k B^H_{\frac{j-1}{m}}\right|^p - c_{k,p} n^{-1} \sum_{i=1}^{[nt]-k+1} \left|u_\frac{i-1}{n}\right|^p \\
&& + \; c_{k,p} \Big(\frac{1}{n}\sum_{i=1}^{[nt]-k+1} \left|u_\frac{i-1}{n}\right|^p - \int^t_0 |u_s|^p ds \Big) \\
& =: & A_t^{(m)} + B_t^{(n,m)} + C_t^{(n,m)} + D_t^{(n)},
\end {eqnarray*}
where $I_n(i) = \{j: \frac{j-1}{m} \in (\frac{i-1}{n}, \frac{i}{n}] \}$, \; $1 \leq i \leq [nt]-k+1$. \\

Because of the stationary property of the increments of $B^H$, the high order difference sequence $\{\Delta_k B^H_{j-1} \,, j\geq 1\}$ is stationary as well. Thus, for any fixed $n \in \mathbb{N}$ and $1 \leq i \leq [nt]-k+1$, we apply the self-similarity property of $B^H$ to scale the high order difference sequence $\{\Delta_k B^H_{(j-1)/m} \,, j \in I_n(i) \}$, and then apply the ergodic theorem to obtain
\begin{equation}\label{uc.erg}
 m^{-1+pH} n \sum_{j \in I_n(i)} \left|\Delta_k B^H_{\frac{j-1}{m}}\right|^p - c_{k,p} \to 0 \,,
\end{equation}
in probability as $m \to \infty$. This implies
\begin{equation}  \label{eq1a}
\lim_{m \to \infty}  \|C^{(n,m)}\|_\infty = 0
\end{equation}
in probability, for any fixed $n \geq 1$. 

For the term $B_t^{(n,m)}$, we apply  arguments similar to those used in the proof of Theorem 1 in \cite{cnw} together with the ergodic theorem  \eqref{uc.erg} for the $k$-th order difference to establish
\begin{equation}  \label{eq2}
\lim_{n \to \infty}\lim_{m \to \infty}\|B^{(n,m)}\|_\infty=0 \, , 
\end{equation}
where the convergence holds in probability.

The term $D^{(n)}_t$ is the remainder of a Riemann sum approximation. For all $p>0$, using the H\"older continuity of $u$, we have 
\begin{equation}  \label{eq3}
  \lim_{n \to \infty} \|D^{(n)}\|_{\infty}=0 \, , 
\end{equation}
almost surely.

It remains to deal with the term   $A^{(m)}$.
We will use the following two elementary inequalities
\begin{eqnarray}
 |x+y+z|^p &\leq& 3^{(p-1)^+}[|x|^p + |y|^p+ |z|^p] \label{iam1}, \\
 ||x|^p - |y|^p| &\leq& (p \vee 1) 2^{(p-2)^+} [|x-y|^p + |y|^{(p-1)^+}|x-y|^{(p \wedge 1)}] \label{iam2}
\end{eqnarray}
for any $p \geq 0$, and any $x, y,z  \in \mathbb{R}$. 
Using inequality (\ref{iam2}), we obtain
\begin {eqnarray}
|A^{(m)}_t|
&\leq&   m^{-1+pH} \sum_{i=1}^{[mt]+1-k} \Big||\Delta_{k}Z_{\frac{i-1}{m}}|^p-|u_{\frac{i}{m}}\Delta_k B^H_{\frac{i-1}{m}}|^p\Big| \nonumber\\
& \leq &  (p\vee1) 2^{(p-2)_{+}} m^{-1+pH}\Bigg\{  \sum_{i=1}^{[mt]+1-k} \Big[\Big|\Delta_{k}Z_{\frac{i-1}{m}} -u_{\frac{i}{m}}\Delta_{k}B^H_{\frac{i-1}{m}}\Big|^p\nonumber \\
&&\qquad + |u_{\frac{i}{m}}\Delta_k B^H_{\frac{i-1}{m}}|^{(p-1)^+}
 \Big|\Delta_{k}Z_{\frac{i-1}{m}}-u_{\frac{i}{m}}\Delta_{k}B^H_{\frac{i-1}{m}}\Big|^{p\wedge1}\Big]\bigg\} \nonumber\\
& =: & (p\vee1)2^{(p-2)_{+}} [E^{(m)}_{k,p} (t)+ F^{(m)}(t)]\,.
\label{e.a^m}
\end {eqnarray}
First, we   use mathematical induction on $k$ to prove $\lim_{m \to \infty}\|E_{k,p}^{(m)} \|_\infty= 0$,  almost surely. 
For $k=1$, the result is true by the proof of Theorem 1 in \cite{cnw}.  Assume the convergence holds true for $k-1$.   We can express $E^{(m)}_{k,p}(t) $  in the following way
\[
E^{(m)}_{k,p}(t)=  m^{-1+pH} \sum_{i=1}^{[mt]+1-k}   \left|  \Phi_{i,1}^{(m)} - \Phi_{i,2}^{(m)} + \Phi_{i,3}^{(m)}\right|^p,
\]
where
\begin{eqnarray*}
 \Phi_{i,1}^{(m)}  &=&  \Delta_{k-1}Z_{\frac{i}{m}} - u_{\frac{i+1}{m}}\Delta_{k-1}B^H_{\frac{i}{m}}, \\
  \Phi_{i,2}^{(m)}  &=&  \Delta_{k-1}Z_{\frac{i-1}{m}}- u_{\frac{i}{m}}\Delta_{k-1}B^H_{\frac{i-1}{m}} ,
  \end{eqnarray*}
  and
  \[
  \Phi_{i,3}^{(m)} =  \Delta_{k-1}B^H_{\frac{i}{m}}(u_{\frac{i+1}{m}}-u_{\frac{i}{m}}).
\]
Then, applying  inequality (\ref{iam1}) yields
\begin{eqnarray*}
E^{(m)}_{k,p} (t)& \leq & 3^{(p-1)^+}m^{-1+pH} \sum_{i=1}^{[mt]+1-k}
\left(  | \Phi_{i,1}^{(m)}|^p +| \Phi_{i,2}^{(m)}|^p + | \Phi_{i,3}^{(m)}|^p \right) \\
& \leq & 3^{(p-1)^+}  \left(2E^{(m)}_{k-1,p} (t)+ m^{-1+pH}   \sum_{i=1}^{[mt]+1-k}|  \Phi_{i,3}^{(m)} |^p \right).
\end{eqnarray*}
Choosing $0 < \epsilon < \frac{1}{q} + H -1$, we can write
\[
 m^{-1+pH}   \sum_{i=1}^{[mt]+1-k}|  \Phi_{i,3}^{(m)} |^p \le  C m^{pH- \frac pq -p(H-\epsilon)}  \|u\|^p_{{1/q}} \|B^H\|^p_{H-\epsilon},
\]
for some constant $C$ depending on $T$, $p$, $\epsilon$, $k$ and $H$. Using the induction hypothesis, and taking into account that $-\frac 1q + \epsilon <H-1<0$, we conclude that  $\|E^{(m)}_{k,p}\|_\infty$ converges to zero almost surely, as $m$ tends to infinity.

Finally, the infinity norm of the term $F^{(m)}$ can be bounded by
\[
\|F^{(m)}\|_{\infty} \leq C \|u \|_{\infty}^{(p-1)^+}  \|B^H\|^{(p-1)^+}_{H-\epsilon} m^{-(p-1)^+(H-\epsilon)} \|E^{(m)}_{k,p\wedge1}\|_{\infty}\,,
\]
where again  $C$  is a constant depending on $T$, $p$, $\epsilon$, $k$ and $H$. Then, $\|F^{(m)}\|_{\infty}$  goes to 0 almost surely, as $m \to \infty$. 

Thus, by \eqref{e.a^m}  we have
 $\|A^{(m)}\|_\infty \to 0$  almost surely, as $m \to \infty$ . The proposition follows then from this convergence and the limits established in  (\ref{eq1a}), (\ref{eq2}) and (\ref{eq3}).
\end {proof}

Next we study the rate of the convergence of \eqref{e.limit-v_n}. We will use the notation
\begin{equation}
v_1^2=\displaystyle\sum^{\infty}_{m=2} \frac{c_m^2}{m!} \Big[1+2\displaystyle\sum_{j=1}^{\infty} \Big(\frac{\rho_{k,H}(j)}{\rho_{k,H}(0)}\Big)^m\Big]\,,\label{e.v1}
\end{equation}
where $c_m = m! (\rho_{k,H}(0))^{\frac{p}{2}} \mathbb{E}[H_m(N)|N|^p]$  and $N$ is a standard Gaussian random variable. We shall first deal with  the
case of the fractional Brownian motion ($Z_t=B^H_t$) and then 
 we shall deal with the general case of stochastic integral.

\begin{pro} \label{pro2}
Fix a positive integer $k \geq 2$. Let  $H \in (0,1)$, $T>0$ and $p>0$. Then
\begin{equation} 
\left(B^H_t, \sqrt{n}\left(n^{-1+pH}V^n_{k,p}(B^H)_t - c_{k,p}  t \right)  \right) \rightarrow
(B^H_t, v_1W_t)\label{e.central-limit-v_n-b}
\end{equation}
in law in the space $\mathcal{D}([0,T])^2 $ equipped with the Skorohod topology, where $v_1$ is defined by \eqref{e.v1} and
$W=\{W_t, t \in [0,T]\}$ is a Brownian motion, independent of the fractional Brownian motion $B^H$.
\end{pro}

\begin{proof}
 The proof will be completed  in two steps.

\noindent {Step 1:} We show the convergence of the finite-dimensional distributions. Let the intervals $(a_l, b_l], l=1,\dots,\nu 
$, be pairwise disjoint in $[0,T]$.
   Define the random vectors $B=(B^H_{b_1}-B^H_{a_1}, \dots, B^H_{b_\nu }-B^H_{a_\nu })$ and $X^{(n)} = (X^{(n)}_1, \dots ,X^{(n)}_\nu )$, where 
   \[
    X^{(n)}_l = n^{-\frac{1}{2}+pH} \sum_{j \in \mathcal{I}_{nl}} \Big|\Delta_kB^H_{\frac{j-1}{n}}\Big|^p - \sqrt{n} c_{k,p} |b_l-a_l|,
   \]
and $\mathcal{I}_{nl}=([na_l]-k+1, [nb_l] - k + 1]$, for $l=1, \dots, \nu $. We claim that 
\begin{equation} \label{bx.fnt}
(B,X^{(n)}) \xrightarrow[n \to \infty]{\mathcal{L}} (B,V) \,,
\end{equation}
where $B,V$ are independent and $V$ is a centered Gaussian vector, whose components  are independent and have variances $v_1^2 |b_l-a_l| $.  Here $v_1^2$ is defined in \eqref{e.v1}.
	
Set $\xi_j = B^H_j - B^H_{j-1}$ and $h(x)=|x|^p - c_{k,p}$. Then $\{\xi_j, j\geq 1\}$ is a stationary Gaussian sequence. Introduce the random vectors
$B^{(n)} = (B^{(n)}_1, \dots ,B^{(n)}_\nu )$ and $Y^{(n)} = (Y^{(n)}_1, \dots ,Y^{(n)}_\nu )$, where
\[
B^{(n)}_l = n^{-H} \sum_{[na_l]<j\leq[nb_l]}\xi_j \,,
\]
\[
Y^{(n)}_l = \frac{1}{\sqrt{n}} \sum_{j \in \mathcal{I}_{nl}} h(\Delta_{k-1}\xi_j) , \, \quad 1 \leq l \leq \nu  \,.
\]
By the self-similarity property of fBm,  the convergence of \eqref{bx.fnt} will follow from the convergence 
\begin{equation}
(B^{(n)},Y^{(n)}) \xrightarrow[n \to \infty]{\mathcal{L}} (B,V) \,.
\label{e.by-to-bv}
\end{equation}
We are going to prove \eqref{e.by-to-bv}  by Theorem \ref{l2.clt}.  Consider the normalized sequence 
\begin{equation}\label{n.seq}
N_j = \frac{\Delta_{k-1}\xi_j}{\sqrt{\rho_{k,H}(0)}} \,,  \quad j\ge 1.
\end{equation}
Since the function $h(x)$ has Hermite rank $2$, the term $Y^{(n)}_l $ can be decomposed as
\[
Y^{(n)}_l = \sum_{m \geq 2} J_m Y^{(n)}_l := \sum_{m \geq 2} \frac{c_m}{\sqrt{n}}\sum_{j \in \mathcal{I}_{nl}} H_m(N_j) \,,
\] 
where $J_m Y^{(n)}_l$ is the projection of  $Y^{(n)}_l$ on the $m$-th Wiener chaos, and
\[
c_m=m!\mathbb{E}[H_m(N)h(\sqrt{\rho_{k,H}(0)}N )]=m!(\rho_{k,H}(0))^{\frac{p}{2}}\mathbb{E}[H_m(N)|N|^p]\,,
\]
with  $N$ being a standard Gaussian random variable. We have the following five statements.

\begin{description}	
\item[(i)] \ $\lim_{n \to \infty} \mathbb{E}[B^{(n)}_h B^{(n)}_l] = \mathbb{E}[(B^H_{b_h}-B^H_{a_h})(B^H_{b_l}-B^H_{a_l})]$\  for  all $1\le h,l \le \nu $. 
\item[(ii)] \ $\mathbb{E} (B^{(n)}_h J_mY_l^{(n)})=0$,  for  all $1\le h,l \le \nu $. This is clear because $B^{(n)}_h \in \mathcal{H}_1$ and $J_mY_l^{(n)} \in \mathcal{H}_m$ with $m \geq 2$.   	
\item[(iii)] \ For all $1 \leq l \leq \nu $, we have
\begin{eqnarray*}
&& \limsup_{n \to \infty} \sum_{m=M+1}^{\infty} \mathbb{E}[|J_m Y_l^{(n)}|^2] = \limsup_{n \to \infty} \sum_{m=M+1}^{\infty} \frac{c_m^2 }{n} \sum_{i, j \in \mathcal{I}_{nl}} \mathbb{E} [H_m(N_i) H_m(N_j)]  \\
&& \le  \limsup_{n \to \infty} \frac{[nb_l] - [na_l]}{n} \sum_{m=M+1}^{\infty} \frac{c_m^2 }{m!} \left[1+2\sum_{i=1}^{[nb_l] - [na_l]}\left|\frac{\rho_{k,H}(i)}{\rho_{k,H}(0)}\right|^m \right] , 
\end{eqnarray*}
which equals the constant $b_l - a_l$ multiplying the tail of $v_1^2$, and it converges to 0 as $M \to \infty$.
\item[(iv)] \ For all $1 \leq l, h \leq \nu $, we have
     \[
	 \mathbb{E}(J_mY_l^{(n)}J_mY_h^{(n)})  =  \displaystyle\frac{c_m^2}{n} \sum_{j \in \mathcal{I}_{nl}} \sum_{i \in \mathcal{I}_{nh}} \mathbb{E}[H_m(N_j)H_m(N_i)].
	 \]
	 As $n \to \infty$, this quantity converges to 
	  \[
	   \Sigma_{lh} = \delta_{lh} c_m^2 (b_l-a_l)\frac{1}{m!} \Big[1+2\displaystyle\sum_{j=1}^{\infty} \Big(\frac{\rho_{k,H}(j)}{\rho_{k,H}(0)}\Big)^m\Big].
	   \] 
\item[(v)] \  For  all $ 1 \leq l, h \leq \nu $, we have
\[
\langle DJ_mY_l^{(n)},DJ_mY_h^{(n)} \rangle_\mathfrak{H} = \frac{c_m^2}{n} \sum_{j \in \mathcal{I}_{nl}} \sum_{i \in \mathcal{I}_{nh}} H_{m-1}(N_j) H_{m-1}(N_i) \mathbb{E}(N_i N_j),
\]
which converges to $m\Sigma_{lh}$ in $L^2(\Omega)$ as $n$ goes to infinity. To show this, we explain the details for $l=h$. The case $l \neq h$ can be  treated in a similar way.
\[
\|DJ_mY_l^{(n)}\|_{\mathfrak{H}}^2 = \frac{c_m^2}{n} \sum_{i \in \mathcal{I}_{nl}} H_{m-1}^2(N_i) +  2 \frac{c_m^2}{n} \sum_{i \in \mathcal{I}_{nl}} \sum_{j=1}^{[nb_l]-[na_l]-1}  H_{m-1}(N_i) H_{m-1}(N_{i+j})\frac{\rho_{k,H}(j)}{\rho_{k,H}(0)}.
\]
Denote $\zeta_i = \sum_{j=1}^{\infty}  H_{m-1}(N_i) H_{m-1}(N_{i+j})\frac{\rho_{k,H}(j)}{\rho_{k,H}(0)}$.  We can show that the sequence $\zeta_i$ converges almost surely and in $L^2(\Omega)$ using the fact that $\sup_j \mathbb{E}\left[\left|H_{m-1}(N_i) H_{m-1}(N_{i+j})\right|^2\right] < \infty$
and $\sum_{j=0}^{\infty}|\rho_{k,H}(j)|^2 < \infty$. Meanwhile, since $N_i$  given by \eqref{n.seq} is stationary  and ergodic       so is $\{\zeta_i, i \geq 1\}$.   By the ergodic theorem, we have thus
in $L^2(\Omega)$
\[
  \lim_{n \to \infty} \|DJ_mY_l^{(n)}\|_{\mathfrak{H}}^2 = c_m^2 (b_l - a_l) \left( \mathbb{E}[H_{m-1}^2(N_1)] + 2 \sum_{j=1}^{\infty} \mathbb{E} [H_{m-1}(N_1) H_{m-1}(N_{1+j})] \frac{\rho_{k,H}(j)}{\rho_{k,H}(0)}\right),
\]
which equals $m \Sigma_{lh}$ for $l=h$.
\end{description}	
These can be used to verify the conditions in  Theorem \ref{l2.clt} to  obtain the convergence $(B^{(n)},Y^{(n)}) \xrightarrow[n \to \infty]{\mathcal{L}} (B,V)$ and correspondingly the convergence \eqref{bx.fnt} stands true.\\

\noindent {Step 2:} \  Let 
\begin{equation}\label{gnt}
	g_n(t)=n^{-\frac{1}{2}+pH}\displaystyle\sum_{j=1}^{[nt]-k+1}|\Delta_{k}B^H_{\frac{j-1}{n}}|^p - \sqrt{n}t c_{k,p} \;.
\end{equation}
We need to show that the sequence of processes $g_n$ is tight in $\mathcal{D}([0,T])$. To this end we want to prove $\mathbb{E} (|g_n(r) - g_n(s)|^2 |g_n(t) - g_n(r)|^2) \leq C(t-s)^2$ for any $s<r<t$.  First, let us compute 
$\mathbb{E}(|g_n(t)-g_n(s)|^4)$ for $s<t$, 
	\[
	 \mathbb{E}(|g_n(t)-g_n(s)|^4) = \frac{1}{n^2} \mathbb{E}\Big(\displaystyle\sum^{[nt]-k}_{j=[ns]-k+1}h(\Delta_{k}B^H_j) + c_{k,p}([nt]-nt-[ns]+ns)\Big)^4  \,.
	\]
Using the elementary inequality $|a+b|^4 \leq 8(|a|^4 + |b|^4)$, we can bound the right-hand side of the above equation as follows
	\begin{eqnarray} \label{g.4m}
		\mathbb{E}(|g_n(t)-g_n(s)|^4) & \leq & \frac{8}{n^2} \mathbb{E}\Big(\Big|\displaystyle\sum^{[nt]-k}_{j=[ns]-k+1}h(\Delta_{k}B^H_j)\Big|^4\Big) + 8c_{k,p}^4 \frac{([nt]-nt-[ns]+ns)^4}{n^2} \nonumber\\
		& \leq & K_1 \frac{([nt]-[ns])^2}{n^2} \Big(\displaystyle\sum^{\infty}_{j=0} \rho_{k,H}^2 (j)\Big)^2 + 8c_{k,p}^4 \frac{([nt]-nt-[ns]+ns)^4}{n^2} \nonumber \\
		& \leq & C \frac{([nt]-[ns])^2}{n^2} + \frac{C}{n^2} \,,
	\end{eqnarray}
where the second inequality follows from  Proposition 4.2 in \cite{ta}. The constant $K_1$ is independent of $n, t, s$, but it may depend on the function $h$ and the distribution of $\Delta_{k}B^H_j$.   

Now for $s<r<t$, if $t-s \geq 1/n$, applying the above inequality \eqref{g.4m}, we have
\begin{eqnarray*}
 \mathbb{E}(|g_n(r)-g_n(s)|^2|g_n(t)-g_n(r)|^2) & \leq & \mathbb{E}(|g_n(r)-g_n(s)|^4 + |g_n(t)-g_n(r)|^4) \\
 & \leq & C \frac{([nt]-[ns])^2}{n^2} + \frac{C}{n^2} \,.
\end{eqnarray*}
Clearly, the right-hand side of the above inequality is at most $C(t-s)^2$.  

If $t-s < 1/n$, then either $s$ and $r$ or $t$ and $r$ lie in the same subinterval $((j-1)/n, j/n]$ for some $j$. It suffices to look at the former case. By \eqref{gnt}, $g_n(r)-g_n(s) = \sqrt{n}c_{k,p} (s-r)$. Using this fact and applying Cauchy-Schwarz inequality, we obtain
\begin{eqnarray*}
 \mathbb{E}(|g_n(r)-g_n(s)|^2|g_n(t)-g_n(r)|^2) & =  & n c_{k,p}^2 (r-s)^2 \mathbb{E}|g_n(t)-g_n(r)|^2 \\
 & \leq & n c_{k,p}^2 (r-s)^2 \sqrt{\mathbb{E}|g_n(t)-g_n(r)|^4} \\
&\leq &C(t-s)^2 \,,
\end{eqnarray*}
where in the last step we have used \eqref{g.4m} for $\mathbb{E}|g_n(t)-g_n(r)|^4$.  The desired tightness property follows from Theorem 13.5 in \cite{bi}.
\end{proof}

\begin{theorem} \label{pro3}
Let  $H \in (0,1)$ and $k \geq 2$.
Fix $p>0$ and suppose $u=\{u_t, t \in [0,T]\}$ is a stochastic process
with H\"{o}lder continuous sample paths  of order $a > {\rm \max}(1-H,\frac{1}{2(p\wedge1)})$ so that  the pathwise Riemann-Stieltjes  integral $Z_t=\int^t_0 u_s dB^H_s$ is well-defined.  Then
\[
(B^H_t, n^{-\frac{1}{2}+pH}V^n_{k,p}(Z)_t-c_{k,p} \sqrt{n} \int^t_0 |u_s|^p ds) \rightarrow
(B^H_t, v_1 \int^t_0|u_s|^p dW_s)\,,
\]
in law in the space $\mathcal{D}([0,T], \RR^2) $
  equipped with the Skorohod topology,  where   $v_1$ is defined by \eqref{e.v1},
$W=\{W_t, t \in [0,T]\}$ is a Brownian motion independent of
the fractional Brownian motion $B^H$.   
\end{theorem}

\begin{proof}
We start  with the following decomposition  of the concerned quantity
\begin{eqnarray*}
& &  n^{-\frac{1}{2}+pH}V^n_{k,p}(Z)_t - c_{k,p} \sqrt{n} \int^t_0 |u_s|^p ds \\
& = & n^{-\frac{1}{2}+pH} \sum^{[nt]+1-k}_{j=1} \Big(\left|\Delta_k Z_{\frac{j-1}{n}}\right|^p - \left|u_{\frac{j}{n}}\Delta_k B^H_{\frac{j-1}{n}}\right|^p\Big) \\
& & + \; \Big(n^{-\frac{1}{2}+pH} \sum^{[nt]+1-k}_{j=1} \left|u_{\frac{j}{n}}\Delta_k B^H_{\frac{j-1}{n}}\right|^p - \frac{c_{k,p}}{\sqrt{n}}\sum^{[nt]+1-k}_{j=1} \left|u_{\frac{j}{n}}\right|^p \Big) \\
& & + \; c_{k,p} \Big(\frac{1}{\sqrt{n}}  \sum^{[nt]+1-k}_{j=1} \left|u_{\frac{j}{n}}\right|^p - \sqrt{n}  \int^t_0 |u_s|^p ds\Big) \\
& =: & A^{(n)}_t + B^{(n)}_t + c_{k,p} C^{(n)}_t \,.
\end{eqnarray*}
Using  the H\"older continuity of $u$,
we can show $\lim_{n \to \infty}||C^{(n)}||_{\infty}=0$ almost surely.  The fact that
$\lim_{n\rightarrow \infty}
||A^{(n)}||_{\infty} = 0$ almost surely can be proved by the same arguments as in the proof of
Theorem \ref{pro1}  under the condition $a > \frac{1}{2(p\wedge1)}$.  It remains to show that  
\begin{equation} \label{t3.bnt}
	B^{(n)}_t \xrightarrow[n \to \infty]{\mathcal{L}} v_1 \int^t_0 |u_s|^p dW_s \,,
\end{equation} 
in the Skorohod topology of $\mathcal{D}([0,T])$.
Denote
\begin{eqnarray*}
g_n(t)&=&n^{-\frac{1}{2}+pH}\sum^{[nt]+1-k}_{i=1}\Big|\Delta_k B^H_{\frac{i-1}{n}}\Big|^p- \frac{[nt]}{\sqrt{n}} c_{k,p},\\
\xi_{j,n} &=& g_n(\frac{j+k-1}{n}) - g_n(\frac{j+k-2}{n}) = n^{-\frac{1}{2}+pH}\Big|\Delta_k B^H_{\frac{j-1}{n}}\Big|^p-\displaystyle\frac{c_{k,p}}{\sqrt{n}}\,.
\end{eqnarray*}
Then $B^{(n)}_t = \sum^{[nt]+1-k}_{j=1} |u_{j/n}|^p \xi_{j,n}$.  In order to finish the proof of \eqref{t3.bnt}, we are going to apply Theorem \ref{wrsum}.  We shall verify  the hypotheses ${\rm (H1)}$ and $\rm{(H2)}$.  By Proposition \ref{pro2} and its proof, $(B^H_t, g_n(t)) \xrightarrow[n \to \infty]{\mathcal{L}} (B^H_t, v_1W_t)$,  so the sequence of processes $\{g_n(t), t\in [0,T]\}$ satisfies the hypothesis ${\rm(H1)}$. Using a similar argument as that for  \eqref{g.4m}, namely by Proposition 4.2 in \cite{ta} again, the family of random variables $\xi$ satisfies the tightness condition $\rm{(H2)}$. This concludes the proof of the theorem.
\end{proof}

\begin{co} \label{co2}
If a stochastic process $\left\{Y_t, t \in [0, T]\right\}$ satisfies $n^{-\frac{1}{2}+pH} V^n_{k,p}(Y)_t \to 0$ uniformly in probability on $[0,T]$ and if  $\left\{Z_t,t \in [0,T]\right\}$ satisfies the conditions of 
Theorem \ref{pro3}, then $$(B^H_t, n^{-\frac{1}{2}+pH}V^n_{k,p}(Y+Z)_t- c_{k,p} \sqrt{n} \int^t_0 |u_s|^p ds) \xrightarrow[n \to \infty]{\mathcal{L}} (B^H_t, v_1 \int^t_0|u_s|^p dW_s)$$
in law in $\mathcal{D} ([0,T])^2$ equipped with the Skorohod topology, where $W=\{W_t, t \in [0,T]\}$ is a Brownian motion independent of
the fractional Brownian motion $B^H$.
\end{co}


\begin{rem}
 When k=1, Theorem  \ref{pro1}, Proposition \ref{pro2}, Theorem \ref{pro3} and Corollary \ref{co2}  are proved in \cite{cnp} and \cite{cnw} for $H \in (0,\frac{3}{4})$.  We need  to use higher order ($k\ge 2$) 
  power variations to estimate the volatility or integrated volatility 
 for a  general Hurst parameter case.
\end{rem}

\section{Estimation of the integrated volatility}\label{s.volatility} 

This section is devoted to the estimation of the integrated volatility
$\int_0^t |\sigma_s|^p ds$ using the $k^{\rm th}$ order power variations. Let the stochastic process $X_t$ satisfy \eqref{e.ou-time-varying}. Motivated by Theorem \ref{pro3}, we construct the $k^{\rm th}$ order power variation estimator $PV_{k,p}(X)_t$ for the integrated volatility $\int_0^t |\sigma_s|^p ds$ as follows
\begin{equation}\label{pvkp}
	PV_{k,p}(X)_t= \displaystyle\frac{n^{-1+pH}V^n_{k,p}(X)_t}{c_{k,p}}, \,\quad  t \in [0,T] \,,
\end{equation}
where the $k^{\rm th}$ order power variation $V^n_{k,p}(X)_t$ is given by (\ref{vkp}), and the normalizing constant $c_{k,p}$ is given by (\ref{ckp}).  For this estimator we have the following asymptotic consistency
and the central limit theorem. 

\begin{thm} \label{isigmaclt}
Let  $X_t$ satisfy \eqref{e.ou-time-varying},  where the sample path of  $\sigma_t$ is H\"older continuous of exponent $a > {\rm max}(1-H, \frac{1}{2 (p \wedge 1)}) $.
Assume $k\geq 2$ and $p>\frac{1}{2}$. Then the estimator $PV_{k,p}(X)_t$ defined by \eqref{pvkp} converges in probability to $\int_0^t |\sigma_s|^p ds$ uniformly on any compact interval $[0,T]$. Furthermore, the following central limit theorem holds true. 
\[
\sqrt{n}\Big(PV_{k,p}(X)_t - \int_0^t |\sigma_s|^p ds\Big) \xrightarrow{\mathcal{L}} \displaystyle\frac{v_1}{c_{k,p}} \int_0^t |\sigma_s|^p d W_s \,,\quad \hbox{as $n \to \infty$}\,,
\]
 in law in $\mathcal{D} ([0,T]) $ equipped with the Skorohod topology, where $v_1$ is defined by \eqref{e.v1} and $W=\{W_t, t \in [0,T]\}$ is a Brownian motion,
   independent of the fractional Brownian motion $B^H$.
\end{thm}

\begin {proof}
By assumption, the stochastic process $\sigma_t$ has H\"older  continuous trajectories of order $a > 1-H$. Then the stochastic process $X_t$ has H\"older  continuous trajectories as well. Write $X_t = X_0 + Y_t + \int_0^t \sigma_s dB_s^H$, where $Y_t = -\theta \int^t_0 X_s ds$. It is easy to check that $n^{-1/2+pH} V^n_{k,p} (Y)_t \to 0$ in probability on $[0,T]$.  The theorem follows from Theorem \ref{pro1},  Theorem \ref{pro3}, and Corollary \ref{co2}. 
\end{proof}

When $\si_t=\si$ is time independent, Theorem \ref{isigmaclt} gives the following result.
\begin{pro}\label{prosigma}
Let $ k\geq 2$ and $p>\frac{1}{2}$.   Then the estimator $PV_{k,p}(X)_t$ converges almost surely to $|\sigma|^pt$ uniformly on any compact interval $[0,T]$.  Furthermore, $\sqrt{n}(PV_{k,p}(X)_t - |\sigma|^pt) \xrightarrow{\mathcal{L}} \displaystyle\frac{v_1 |\sigma|^p}{c_{k,p}} W_t$ as $n \to \infty$ in law in $\mathcal{D} ([0,T]) $ equipped with the Skorohod topology,  where $v_1$ is given by
\eqref{e.v1} and $W_t$ is a Brownian motion  independent of the fractional Brownian motion $B^H$.
\end{pro}
This proposition  gives another estimator for $\si$:  
\begin{equation} \label{sigma.h}
  {|\hat\sigma}_T|^p = \displaystyle\frac{n^{-1+pH}V^n_{k,p}(X)_T}{c_{k,p}T}.
\end{equation}
It is easy to see that Theorem \ref{isigmaclt} and Proposition \ref{prosigma} yield the following result. 
\begin{pro} \label{sigmaclt-constant}
When $H \in (0, \frac{3}{4})$, set $k \geq 1$.  When $H \in [\frac{3}{4},1)$, set $k\geq 2$. Assume $p>\frac{1}{2}$.  Then, 
 the estimator $ {|\hat\sigma}_T|^p$  defined by \eqref{sigma.h} converges almost surely to $|\sigma|^p$.  Furthermore,  $\sqrt{n}({|\hat\sigma}_T|^p - |\sigma|^p) \xrightarrow{\mathcal{L}} N(0, \nu^2)$ as $n \to \infty$,  where 
  the asymptotic variance $\nu^2$  is  given by
\begin{equation}
\nu^2=  \frac{\Gamma(\frac{1}{2})^2}{2^p \Gamma(\frac{p+1}{2})^2} \sum_{m=2}^{\infty} m! \mathbb{E}^2(H_m(N)|N|^p)\Big[1+2\sum_{j=1}^{\infty} \Big(\frac{\rho_{k,H}(j)}{\rho_{k,H}(0)}\Big)^m\Big] \frac{\sigma^{2p}}{T}\,. \label{e.nu2} 
\end{equation}
Here $N$ is a standard Gaussian random variable.\\
\end{pro}
 
Usually the variance in \eqref{e.nu2}  is complicated to compute. 
When $p=2$, we compute  the normalized asymptotic variance of $\nu^2\frac{T}{\sigma^{2p}}$
for some $H$ and $k$ in the following Table 1.   \\

\centerline{Table 1: Normalized Asymptotic variance $\nu^2 \frac{T}{\sigma^{2p}}$ (when $p=2$)}
\centerline{
\begin{tabular}{c|c|c|c|c|c}
	\hline
	& \multicolumn{5}{|c}{$k$} \\
	\cline{2-6}
	\quad $H$ \qquad  &  1	&	2	&	3	&	4		&	5   \\
	\hline
  \quad 0.1 \qquad & \; 2.7283 \; & \; 3.7127 \;  & \; 4.4814 \; & \; 5.1354 \; & \; 5.7147 \; \\
  \quad 0.3 \quad &   2.2504  &  3.3539  &  4.1909 &   4.8855  &  5.4924\\
  \quad 0.5 \quad &  2.0000  &  3.0000  &  3.8889 &   4.6200  &  5.2531\\
  \quad 0.6 \quad &  2.1639  &  2.8308  &  3.7364 &   4.4830  &  5.1282\\
  \quad 0.7 \quad &   3.6088  &  2.6704  &  3.5846 &   4.3443  &  5.0005\\
  \quad 0.8 \quad &  - & 2.5215 & 3.4348 & 4.2047 & 4.8707 \\
  \quad 0.9 \quad &  - & 2.3872 & 3.2884 & 4.0651 & 4.7393 \\
  \hline
\end{tabular}
}
\medskip

We   see that when $H$ is small (for example when $H\le 0.6$), 
it is more efficient  to use the first order power variation 
than the higher order ones.  However,  when $H$ is large (for example  when $H\ge \frac{3}{4}$),  
the central limit theorem of the first order power variation does not hold, but we always have the central limit theorem for the second order power variation.
As long as the central limit theorem of the power variation holds, it is preferable to use the lowest order.

\section{Estimation of the drift}\label{s.drift}
In this section  we assume that the volatility $\sigma$ is known  and we want to
estimate the drift parameter $\theta$. There have been two popular types of estimators for this drift parameter.
One is the maximum likelihood estimator and the  other one is the least square estimator.  In the Brownian motion case, they coincide, but for the fractional Ornstein-Uhlenbeck processes they are different 
(see \cite{hn}  and  \cite{kl}).   
We shall focus on  the least square estimator as introduced in \cite{hn}:
\begin{equation}
\hat\theta_T=-\displaystyle\frac{\int^T_0 X_t dX_t}{\int^T_0 X_t^2 dt} = \theta -\sigma \displaystyle\frac{\int^T_0 X_t dB^H_t}{\int^T_0 X_t^2 dt} \, , \label{thetaT}
\end{equation}
where $dB_t^H$ denotes the divergence integral. In the paper \cite{hn}, the almost sure convergence of $\hat\theta_T$ to $\theta$ is proved for $H \geq \frac{1}{2}$ and the central limit theorem is obtained for $H \in [\frac{1}{2}, \frac{3}{4})$. In this paper, we shall extend these results for a general Hurst parameter $H \in (0,1)$. In addition, we shall also consider a less popular but very robust estimator: ergodic type estimator.  
  
To simplify notation, we assume $X_0=0$. In this case the solution to \eqref{e.ou-time-varying} is   given by
\begin{equation} \label{xt.solution}
X_t=\sigma \int^t_0 e^{-\theta(t-s)}dB^H_s \,.
\end{equation}

\begin{thm}
  For $H \in (0,1)$, $\hat\theta_T \rightarrow \theta$ a.s. as $T\rightarrow \infty$.
\end{thm}

\begin{proof}
 Using integration by parts, we can write
	\begin{equation}
	 X_t = \sigma \displaystyle\int_0^t e^{-\theta(t-s)}dB_s^H = \sigma \Big(B_t^H- \theta \int_0^t B_s^H e^{-\theta(t-s)} ds\Big) \,.   \label{t4-1-1}
	\end{equation}
Since $X_t$ is in the first Wiener chaos, we have the relationship between the divergence integral and the Stratonovich integral as

\begin{equation}\label{t4-1-2}
 \int^T_0 X_t  dB_t^H   =   \int^T_0 X_t \circ dB_t^H - \ell(T) \,,
\end{equation}
where $\ell(T) = \mathbb{E} \int^T_0 X_t \circ dB_t^H$. Using \eqref{t4-1-1}, $\ell(T)$ can be computed as follows
\begin{eqnarray}
  \ell(T) & = & \sigma \mathbb{E} \int_0^T (B_t^H - \theta \int_0^t B_s^H e^{-\theta (t-s)}ds) \circ dB_t^H \nonumber \\
  & = & \sigma \Big[\frac{1}{2}T^{2H} - \theta \int_0^T \int_0^t e^{ - \theta (t-s)} \frac{\partial \mathbb{E}(B_s^H B_t^H)}{\partial t} ds dt \Big] \nonumber \\
  &= &  \frac{\sigma}{2}T^{2H} - m(T) \label{t4-1-3}\,,
\end{eqnarray}
where 
\[
  m(T) := H \theta \sigma \int_0^T \int_0^t e^{-\theta (t-s)} (t^{2H-1} - (t-s)^{2H-1}) dsdt.
\]

Making the substitutions $t-s \to u$, $s \to v$ and then integrating first  in the variable $v$   yield 
\begin{equation}
 m(T) = \frac{\sigma}{2} \gamma_{\theta T}^1 T^{2H} + \sigma \theta^{-2H} \gamma_{\theta T}^{2H+1} (H-\frac{1}{2}) - TH\sigma \theta^{1-2H} \gamma_{\theta T}^{2H} \,.
\end{equation}

In the above equation, we use the notation $\gamma_T^{\alpha} = \int_0^T e^{-x}x^{\alpha-1}dx$. Observe that $\gamma_T^{\alpha}$ converges to $\Gamma(\alpha)$ exponentially fast as $T \to \infty$.  Then clearly we have 
\begin{equation}
  \lim_{T \to \infty} T^{-1} \ell(T) = \lim_{T \to \infty} T^{-1}\big(\frac{\sigma}{2}T^{2H}-m(T)\big) = H\sigma \theta^{1-2H}\Gamma(2H) \,. \label{limit_mT}
\end{equation}

On the other hand, we have 
\begin{equation}
\sigma \int^T_0 X_t \circ dB_t^H = \int^T_0 X_t \circ (dX_t + \theta X_t dt) = \frac{X_T^2}{2} + \theta \int^T_0 X_t^2 dt \label{t4-1-4}\,. 
\end{equation}
Combining (\ref{t4-1-2}) and (\ref{t4-1-4}) we obtain  
\begin{equation}
 \sigma \int^T_0 X_t dB_t^H 
 = \frac{X_T^2}{2} + \theta \int^T_0 X_t^2 dt - \sigma \ell(T) \,.\\
 \label{t4-1-5}
\end{equation}

From Lemma \ref{yt.lim}, we see $\displaystyle\lim_{T \to \infty} \frac{X_T^2}{T} = 0$. Therefore, by Lemma \ref{xt2}, \eqref{limit_mT}, and \eqref{t4-1-5}, we have
\[
 \lim_{T \to \infty} T^{-1} \sigma \int_0^T X_t dB_t^H = 0.
\]
As a consequence,
\[
\lim_{T \to \infty} \hat\theta_T = \displaystyle\lim_{T \to \infty} \Big(\theta - 
\frac{\sigma \int_0^T X_t dB_t^H}{\int_0^T X_t^2 dt}\Big)= \theta \, .
\]
\end{proof}
   
The next theorem shows the asymptotic laws for the least square estimator $\hat\theta_T$. 
\begin{thm} \label{thetaclt}  As $T\rightarrow \infty$, the following convergence results hold true. 
  \begin{enumerate}
    \item[(i)] For $H \in (0,\frac{3}{4})$, $\sqrt{T}(\hat\theta_T-\theta) \xrightarrow{\mathcal{L}} N(0,\theta\sigma^2_H)$, where
	\[
	\sigma^2_H = \begin{cases} 
	(4H-1)+ \frac{2\Gamma(2-4H)\Gamma(4H)}{\Gamma(2H)\Gamma(1-2H)} &\mbox{when } H \in (0, \frac{1}{2}) \;, \\
	 \\
	(4H-1)\Big[1+\frac{\Gamma(3-4H)\Gamma(4H-1)}{\Gamma(2-2H)\Gamma(2H)}\Big] & \mbox{when } H \in [\frac{1}{2}, \frac{3}{4}) \;.
	\end{cases}\\
	\]
    \item[(ii)] For $H=\frac{3}{4}$, $\displaystyle\frac{\sqrt{T}}{\sqrt{\log(T)}}(\hat\theta_T-\theta)\xrightarrow{\mathcal{L}} N(0,4\pi^{-1}\theta).$ \\
    \item[(iii)] For $H \in (\frac{3}{4},1)$, $T^{2-2H}(\hat\theta_T-\theta) \xrightarrow{\mathcal{L}} \displaystyle\frac{-\theta^{2H-1}}{H\Gamma(2H)}R_1$, where $ R_1=I_2(\delta_{0,1}) $ is the Rosenblatt random variable and $\delta_{0,1}$ is the Dirac-type distribution   defined in \eqref{ddel.def}.\\
  \end{enumerate}
\end{thm}

\begin{rem}
It is interesting to note that when $H \in (0,\frac{1}{2})$, by the fact $\lim_{z \to 0} z\Gamma(z) = 1$, we have
 \[
   \lim_{H \to \frac{1}{2}^-} \sigma_H^2 = 2
 \]
which is consistent with $\sigma_H^2 = 2$ if $H = \frac{1}{2}$. Moreover, we also see that $\lim_{H \to 0} \sigma_H^2 = 0$.
\end{rem}

\begin {proof} 

The case $H \in [\frac{1}{2}, \frac{3}{4})$ was proved in \cite{hn}. We shall use Malliavin calculus to prove the theorem for $H \in (0,\frac{1}{2})\cup [\frac{3}{4},1)$.\\  

\noindent  {\bf Step 1:} We use Theorem \ref{fm.theorem} to prove the central limit theorem when $H \in (0, \frac{1}{2})$. By \eqref{thetaT} and \eqref{xt.solution}, we can write our target quantity as
  \begin{equation}\label{e.thetat-theta} 
  \sqrt{T}(\hat\theta_T-\theta) = - \frac{\frac{\sigma^2}{\sqrt{T}} \int^T_0(\int^t_0 e^{-\theta(t-s)}dB^H_s)dB^H_t}{\int^T_0X_t^2dt / T} = \frac{-\frac{\sigma^2}{2\sqrt{T}} F_T}{\int^T_0 X_t^2 dt / T}\,,
  \end{equation}
 where
\begin{equation}
F_T = \int^T_0\int^T_0 e^{-\theta |t-s|} dB^H_s dB^H_t\,.
\label{e.ft1} 
\end{equation}
We introduce the function 
\[
f(s,t)=\frac{1}{\sqrt{T}}e^{-\theta|s-t|}\mathbf{1}_{[0,T]^2} \, .
\]
Then $\frac{1}{\sqrt{T}}F_T = I_2(f)$ is in the second Wiener chaos. Our main objective  is to use Theorem \ref{fm.theorem} to obtain the central limit theorem for the term $\frac{1}{\sqrt{T}}F_T$ and then  we   apply Lemma \ref{xt2} and Slutsky's theorem for \eqref{e.thetat-theta} to obtain the central limit theorem of $\hat\theta_T$. First of all, let us check the variance assumption in Theorem \ref{fm.theorem}.  By the isometry between the Hilbert space $\mathfrak{H}^{\otimes 2}$ and the second chaos $\mathcal{H}_2$, we have 
  \begin{equation*}
   \mathbb{E}\left(\frac{1}{T} F_T^2\right) = \frac{2}{T} \langle e^{-\theta|s_1-t_1|}, e^{-\theta|s_2-t_2|} \rangle_{\mathfrak{H}\otimes \mathfrak{H}}.
  \end{equation*}
To compute the above norm, we shall use the definition of the tensor product space where the norm in the Hilbert space $\mathfrak{H}$ is defined by (\ref{hinnerp}), namely,
  \begin{equation}
   \mathbb{E}\left(\frac{1}{T} F_T^2\right)  = \frac{2}{T} \int_{[0,T]^4} \frac{\partial e^{-\theta|s_1-t_1|}}{\partial t_1} \frac{\partial e^{-\theta|s_2-t_2|}}{\partial s_2}\frac{\partial R_H(s_1,s_2)}{\partial s_1}\frac{\partial R_H(t_1,t_2)}{\partial t_2}ds_1ds_2dt_1dt_2. \label{ET2}
  \end{equation}
By Equation \eqref{e.lemm5.1} in Lemma \ref{lem5.1}, we have
\begin{equation} \label{ft.v}
\lim_{T\to\infty} \mathbb{E}\left(\frac{1}{T} F_T^2\right)=4H^2 \theta^{1-4H} \Gamma(2H)^2 \Big((4H-1) + \frac{2\Gamma(2-4H)\Gamma(4H)}{\Gamma(2H)\Gamma(1-2H)}\Big) \,.
\end{equation}
Next, let us check  the second condition in Theorem \ref{fm.theorem}.  The first contraction of the kernel $f$ is
  \begin{equation}
	f \otimes_1 f := g(s,t) = \frac{1}{T} \langle e^{-|\cdot-s|}\mathbf{1}_{[0,T]}(\cdot), e^{-|\cdot-t|}\mathbf{1}_{[0,T]}(\cdot) \rangle_\mathfrak{H} \,. \label{e.def-g} 
  \end{equation}
We want to prove that the norm of the function $g(s,t)$ in the Hilbert space $\mathfrak{H}^{\otimes{2}}$ goes to 0 as $T \to \infty$.  Using the identity  (\ref{sfbm}),  we rewrite
  \begin{eqnarray*}
   g(s,t) &=& \frac{1}{Tc_H^2} \int_{\mathbb{R}} \mathcal{F} (e^{-|\cdot-s|}\mathbf{1}_{[0,T]}(\cdot))(\xi) \overline{\mathcal{F} (e^{-|\cdot-t|}\mathbf{1}_{[0,T]}(\cdot))(\xi)} |\xi|^{1-2H} d\xi\\
	      &=& \frac{4}{Tc_H^2} \int_{\mathbb{R}} \Big(\int_{\mathbb{R}} \frac{e^{-is\eta}}{1+\eta^2} \cdot \frac{1-e^{-iT(\xi-\eta)}}{i(\xi-\eta)}d\eta\Big)
		                           \Big(\int_{\mathbb{R}} \frac{e^{it\eta'}}{1+\eta'^2} \cdot \frac{1-e^{iT(\xi-\eta')}}{-i(\xi-\eta')}d \eta'\Big) |\xi|^{1-2H} d\xi .
  \end{eqnarray*}
Observe that $g(s,t)$ is the inverse Fourier transformation of the following  function  
$$h(s, t)= \frac{4}{Tc_H^2}\int_{\mathbb{R}} \Big(\frac{1}{1+s^2} \cdot \frac{1-e^{-iT(\xi+s)}}{i(\xi+s)}\Big)\Big(\frac{1}{1+t^2} \cdot \frac{1- e^{iT(\xi-t)}}{-i(\xi-t)}\Big) |\xi|^{1-2H} d\xi.$$  
By the Parseval's identity, the norm of the function $g$ in the space $\mathfrak{H}^{\otimes 2}$ can be computed as
  \begin{eqnarray}
  ||g||_{\mathfrak{H}^{\otimes 2}}^2 & = & \frac{1}{c_H^2}\int_{\mathbb{R}^2}|h(\eta,\eta')|^2 |\eta|^{1-2H} |\eta'|^{1-2H} d\eta d\eta' \nonumber \\
                                    & \leq & \frac{C}{T^2} \int_{\mathbb{R}^2} \frac{|\eta|^{1-2H}}{(1+\eta^2)^2}\frac{|\eta'|^{1-2H}}{(1+\eta'^2)^2} \Big(\int_{\mathbb{R}} \frac{|e^{-iT(\xi-\eta)}-1|}{|\xi-\eta|} \frac{|e^{-iT(\xi-\eta')}-1|}{|\xi-\eta'|} |\xi|^{1-2H} d\xi \Big)^2 d \eta d \eta'\,. \nonumber\\\label{g.ineq}
  \end{eqnarray} 
Now our task is to show the right-hand side of the above inequality goes to 0 as $T \to \infty$. Split $\mathbb{R}$ into $\mathbb{R} = \mathscr{I}_1 \cup \mathscr{I}_2$ where $\mathscr{I}_1 = \{\xi: |\xi|<|\eta|+|\eta'|\}$ and $\mathscr{I}_2 = \{\xi: |\xi|\geq|\eta|+|\eta'|\}$. Put $\mathscr{I}_j \cap \mathbb{R}_+ = \mathscr{I}^+_j$ for $j=1,2$. Denote the functions
\[
f(\eta, \eta') = \frac{\eta^{1-2H}}{1+\eta^4}\frac{\eta'^{1-2H}}{1+\eta'^4}, \quad \eta, \eta' \in \mathbb{R_+} \,,
\]
and 
\[
  f_{\alpha_j}(\xi, \eta, \eta') = |\xi - \eta|^{\alpha_j-1} |\xi-\eta'|^{\alpha_j-1} |\xi|^{1-2H} , \quad \xi, \eta, \eta' \in \mathbb{R}
\]
for $j=1,2$. We shall use the inequality $|e^{-iTx}-1| \leq C_{\alpha} T^{\alpha} |x|^{\alpha}$ for any $0 < \alpha < 1$ and some constant $C_{\alpha}>2/\alpha$ to bound the corresponding factors in \eqref{g.ineq}. The choices of $\alpha$ are different on $\mathscr{I}_j$. Namely, we choose $ 1/4<\alpha_1 < 1/2$ on $\mathscr{I}_1$ and $H/2<\alpha_2<H$ on $\mathscr{I}_2$. In this way, we obtain
  \begin{eqnarray}
   ||g||_{\mathfrak{H}^{\otimes 2}}^2 & \leq & CT^{-2} \int_{\mathbb{R}^2} f(|\eta|,|\eta'|)  \Big(\sum_{j=1}^2 \int_{\mathscr{I}_j} T^{2\alpha_j}f_{\alpha_j}(\xi,\eta,\eta')d\xi \Big)^2 d\eta d\eta' \nonumber \\
   & \leq & CT^{-2} \int_{\mathbb{R}^2} f(|\eta|,|\eta'|) \Big(\sum_{j=1}^2 \int_{\mathscr{I}_j}T^{2\alpha_j}f_{\alpha_j}(|\xi|,|\eta|,|\eta'|)d\xi\Big)^2 d\eta d\eta' \nonumber \\
   & \leq & CT^{-2} \int_{\mathbb{R}^2_+} f(\eta,\eta')  \Big(\sum_{j=1}^2 \int_{\mathscr{I}^+_j}T^{2\alpha_j}f_{\alpha_j}(\xi,\eta,\eta')d\xi \Big)^2 d\eta d\eta' \label{gnorm}\,.
  \end{eqnarray}
By the symmetry of $\eta, \eta'$ in $f_{\alpha_j}(\xi,\eta,\eta')$, we can assume that $\eta \geq \eta'$.  Then by Lemma \ref{fxi}, for any $\epsilon$ satisfying $0 < \epsilon < \alpha_1 \wedge (2\alpha_1 - 1/2)$, there exist some positive constants $K_1, K_2$ depending on $\alpha_1, \epsilon$ such that
  \begin{equation}
   \int_{\mathscr{I}^+_1}f_{\alpha_1}(\xi,\eta,\eta') d\xi \leq K_1(\eta-\eta')^{-1+2\alpha_1 -\epsilon} \eta^{1-2H+\epsilon} + K_2 (\eta-\eta')^{-1+2\alpha_1} \eta^{1-2H} \label{f.alpha1}\,.
   \end{equation}
Similarly there exist some positive $K_3, K_4$, depending on $\alpha_2$ and $H$, such that
  \begin{equation}
   \int_{\mathscr{I}^+_2}f_{\alpha_2}(\xi,\eta,\eta') d\xi \leq K_3 (\eta')^{2\alpha_2 -2H} + K_4 \eta^{1-2H} (\eta')^{-1+2\alpha_2}\label{f.alpha2} \,.
  \end{equation} 
Denote
\[
 H_{\lambda_1,\lambda_2,\lambda_3} (\eta,\eta') = f(\eta,\eta')  (\eta')^{2\lambda_1} \eta^{2\lambda_2} (\eta-\eta')^{2\lambda_3} , \quad \eta \geq \eta' > 0 \,,
\]
where $\lambda_1 \in [-1+2\alpha_2,0], \quad \lambda_2 \in [0, 1-2H+\epsilon], \quad \lambda_3 \in [-1+2\alpha_1-\epsilon,0]$.  We substitute \eqref{f.alpha1}, \eqref{f.alpha2} into \eqref{gnorm} and apply the elementary inequality $(\sum_{i=1}^4 a_i)^2 \leq 4 \sum_{i=1}^4 a_i^2 $ to obtain 
\[
  ||g||_{\mathfrak{H}^{\otimes 2}}^2 \leq 4CK^2 T^{-2+4(\alpha_1 \vee \alpha_2)} \int_{\mathbb{R}_+^2} \sum_{(\lambda_1,\lambda_2,\lambda_3) \in \Pi} H_{\lambda_1,\lambda_2,\lambda_3} (\eta,\eta') d\eta d\eta' \,,
\]
where $K=\max\{K_i,\, i=1,...,4\}$ and $\Pi=\{(0,1-2H+\epsilon,-1+2\alpha_1-\epsilon),(0,1-2H,-1+2\alpha_1),(2\alpha_2-2H,0,0),(-1+2\alpha_2, 1-2H,0)\}$.
The functions $H_{\lambda_1,\lambda_2,\lambda_3} (\eta,\eta')$ are integrable on $\mathbb{R}_+^2$. Thus,
\[
\lim_{T \to \infty}||g||_{\mathfrak{H}^{\otimes 2}}^2 = 0 \,.
\]
By Theorem \ref{fm.theorem}, as $T$ goes to infinity, the term $\frac{1}{\sqrt{T}}F_T$ converges in distribution to a  centered Gaussian random variable with variance given by \eqref{ft.v}.  Applying Slutsky's theorem and Lemma \ref{xt2} to the equation \eqref{e.thetat-theta}, we finish the proof of the theorem when $H \in (0, \frac{1}{2})$. \\

\noindent  {\bf Step 2:} The case $H=\frac{3}{4}$ can be dealt with in a similar way as in the proof of Theorem 3.4 in \cite{hn}.  But now we need to use  Lemma \ref{lem5.1}.\\

\noindent   {\bf Step 3:} In this step we will prove the theorem when $H \in (\frac{3}{4}, 1)$. Recall that the term $F_T$ is given by \eqref{e.ft1}. By \eqref{thetaT} and \eqref{xt.solution}, we write
\[
  T^{2-2H}(\hat\theta_T-\theta)=\frac{-\frac{\sigma^2}{2} T^{1-2H} F_T}{\int^T_0 X_t^2 dt / T}\,.
\]
Denote
\begin{equation} \label{e.ft2}
\widetilde F_T = T^{2H} \int_{[0,1]^2} e^{-\theta T |t-s|} dB^H_s dB^H_t \,.
\end{equation}
By the self-similarity property of the fBm, the process $\{F_T, T>0\}$ has the same law as $\{\widetilde F_T, T>0\}$. To prove part (iii) of the theorem, we need to show 
 $T^{1-2H} F_T \xrightarrow{\mathcal{L}} 2\theta^{-1}R_1$. It suffices to prove
 \begin{equation}
 \displaystyle\lim_{T \to \infty}\mathbb{E}(T^{1-2H}\widetilde F_T-2\theta^{-1}R_1)^2=0\,.
 \label{e.thm5.6-toprove}
 \end{equation} 
By Equations \eqref{e.lemm5.1-3} and  \eqref{e.lemm5.1-4}, we see immediately that
\[
  \lim_{T \to \infty}\mathbb{E}\left(T^{2-4H} \widetilde F_T^2\right) = \lim_{T \to \infty}\mathbb{E}\left(T^{2-4H} F_T^2\right) = \frac{16\alpha_H^2 \theta^{-2}}{(4H-2)(4H-3)}\,,
\]
\[
  \lim_{T \to \infty} \mathbb{E}\left[2\theta^{-1} R_1 (T^{1-2H} \widetilde F_T)\right]=\frac{16\alpha_H^2 \theta^{-2}}{(4H-2)(4H-3)} \,,
\]
where $\alpha_H=H(2H-1)$. On the other hand, we have 
\begin{eqnarray*}
\mathbb{E}(2\theta^{-1}R_1)^2
&=&8\theta^{-2}\alpha_H^2 \int_{[0,1]^4} \delta_{0,1}(s-t) \delta_{0,1}(s'-t') |s-s'|^{2H-2}|t-t'|^{2H-2}dsdtds'dt'\\
&=&8\theta^{-2}\alpha_H^2 \int_{[0,1]^2} |t-s|^{4H-4} dsdt = \frac{16\theta^{-2}\alpha_H^2}{(4H-3)(4H-2)} \;.
\end{eqnarray*}
This shows \eqref{e.thm5.6-toprove} and hence   completes the proof
of the theorem.
\end {proof}

\begin{thm} \label{tthetaclt}
Define an ergodic-type  estimator for the drift parameter by 
\begin{equation}
\widetilde{\theta}_T = \Big ( \frac{1}{\sigma^2 H\Gamma(2H)T} \int^T_0 X^2_t dt \Big)^{-\frac{1}{2H}}\,.
\end{equation}
Then $\widetilde{\theta} _T\rightarrow \theta$ almost surely as $T \to \infty$.  Furthermore,  we have 
the following central limit theorem ($H\le 3/4$) and noncentral limit theorem  ($H> 3/4$).
\begin{enumerate}
 \item [(1)] When  $H \in (0,\frac{3}{4})$, we have  $\sqrt{T} (\widetilde{\theta}_T - \theta) \xrightarrow{\mathcal{L}} N(0,\frac{\theta}{(2H)^2} \sigma_H^2)$ as $T \to \infty$, where $\sigma_H^2$ is  defined in  Theorem \ref{thetaclt}.
 \item [(2)] When  $H=\frac{3}{4}$,  we have 
 $\frac{\sqrt{T}}{\log(T)} (\widetilde{\theta}_T - \theta) \xrightarrow{\mathcal{L}} N(0,\frac{16\theta}{9\pi}) $ as $T \to \infty$.
 \item [(3)] When  $H \in (\frac{3}{4},1)$,  we have $T^{2-2H} (\widetilde{\theta}_T - \theta) \xrightarrow{\mathcal{L}} \frac{-\theta^{2H-1}}{H \Gamma(2H+1)} R_1$, where $R_1=I_2(\delta_{0,1})$ is the Rosenblatt random variable, and $\de_{0,1}$ is the Dirac-type function defined in \eqref{ddel.def}.
\end{enumerate}
\end{thm}

\begin{proof}
The paper \cite{hn}   provides a  proof of  the theorem 
when   $H \in (\frac{1}{2}, \frac{3}{4})$.  Here we present a  proof valid for all $H \in (0,1)$. 
By Lemma \ref{xt2}, it is easy to see $\widetilde{\theta}_T \rightarrow \theta$ almost surely as $T \to \infty$.

  We prove the central limit theorem when $H \in (0, \frac{3}{4})$.  For $H \in [\frac{3}{4}, 1)$, the proof is similar.  By \eqref{thetaT} and  \eqref{t4-1-5},  we can derive an expression for $\int_0^T X_t^2 dt$, and then
express $\widetilde{\theta}_T $ as a function of $\hat\theta_T$. In this way, we obtain
 \begin{eqnarray*}
  \sqrt{T}(\widetilde{\theta}_T-\theta) & = & \sqrt{T} \Bigg[\Bigg(\frac{\sigma^2 H \Gamma(2H) \hat\theta_T}{-\frac{X_T^2}{2T}+\sigma T^{-1} \ell(T)}\Bigg)^{\frac{1}{2H}}-\theta\Bigg]\,. 
 \end{eqnarray*}
 By  Lemma \ref{yt.lim} and (\ref{limit_mT}) we have 
  \begin{eqnarray*}
  \sqrt{T}(\widetilde{\theta}_T-\theta) 
  & = & \sqrt{T} \Big[\Big(\frac{1}{\theta^{1-2H}+o(T^{-1/2})}\Big)^{\frac{1}{2H}} \hat\theta_T^{\frac{1}{2H}} - \theta \Big] \\
  & = & \sqrt{T} \theta^{1-\frac{1}{2H}}(\hat\theta_T^{\frac{1}{2H}}-\theta^{\frac{1}{2H}})+\sqrt{T} \; o(T^{-1/2})\hat\theta_T^{\frac{1}{2H}} \, .
 \end{eqnarray*} 
Meanwhile, we can write
 \begin{equation*}
   \sqrt{T}\Big[\hat\theta_T^{\frac{1}{2H}} - \theta^{\frac{1}{2H}}\Big] = \sqrt{T} \Big[\frac{1}{2H}\theta^{\frac{1}{2H}-1}(\hat\theta_T - \theta) + \frac{1-2H}{8H^2}(\hat\theta_T - \theta)^2 (\theta_T^*)^{\frac{1}{2H}-2}\Big]
 \end{equation*}
 for some $\theta_T^*$ between $\theta$ and $\hat\theta_T$.  
Now the theorem follows from  Theorem \ref{thetaclt}.
\end{proof}

\begin{rem}
  By the   property for gamma function:  $\Gamma(1-z)\Gamma(z)=\frac{\pi}{\sin(\pi z)}$ for $z \notin \mathbb{Z}$, we see  $\lim_{H \to 0}\frac{\theta}{(2H)^2} \sigma_H^2 = \frac{\pi^2}{2} \theta$.
\end{rem}

Now we have obtained the asymptotic law of the least square estimator (LSE) $\hat\theta_T$ and the ergodic type estimator (ETE) $\widetilde{\theta}_T$. Next,  we compare these two estimators with the maximum likelihood estimator  by computing their asymptotic variance.  For convenience, we  assume $\theta=1$. As it  can be seen from Figure 1, the asymptotic variance of LSE increases as $H$   increases. When $H \in (0,\frac{1}{2})$, the asymptotic variance of LSE is less than that of MLE, where the converse is true for $H \in (\frac{1}{2}, \frac{3}{4})$. The asymptotic variance of ETE decreases on $H \in (0,\frac{1}{2})$ and then increases on $H \in (\frac{1}{2}, \frac{3}{4})$; however, it does not blow up as fast as LSE does when $H$ is close to $\frac{3}{4}$. If we justify these three estimators only based on asymptotic variance, LSE performs best when $H \in (0,\frac{1}{2})$ and MLE performs best when $H \in (\frac{1}{2}, \frac{3}{4})$.  At $H=\frac{1}{2}$, these three estimators have the same asymptotic variance.  

\begin{figure}[h]
\centering
\vspace*{-3cm}
\includegraphics[width=0.75\textwidth,height=14cm]{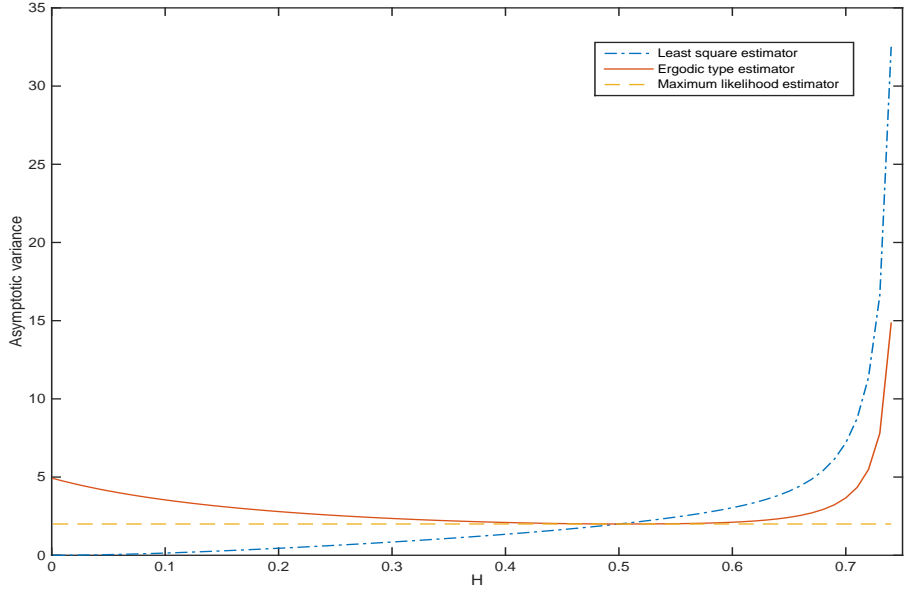}
\vspace*{-2.5cm}
\caption{Asymptotic Variance of the Three Estimators}
\end{figure}
	
The estimators $\hat \theta_T$ and $\widetilde{\theta}_T$ are based on continuous time data. In practice  the process can only be observed at discrete time instants. This motivates us to construct an estimator based on discrete observations. We assume that the fractional Ornstein-Uhlenbeck process $X$ given by \eqref{xt.solution} can be observed at discrete time points $\{t_k = kh, k=0, 1, \dots, n\}$. We shall use $nh$ instead of $T$ for the time period of the observation.  Here $h$ represents the observation frequency and it depends on $n$. We will only consider the high frequency observation case, namely, we shall assume that $h\rightarrow 0$ as $n\rightarrow \infty$. We shall use ergodic type estimator  since it can be expressed as a  pathwise Riemann integral with respect to time.  
The following Theorem shows its asymptotic consistency and some results on its asymptotic law.

\begin{thm} \label{disctheta}
	Assume the fractional Ornstein-Uhlenbeck process $X$ given by \eqref{xt.solution} is observed at discrete time points $\{t_k = kh, k=0, 1, ..., n\}$. Suppose that  $h$ depends on $n$ and as $n \to \infty$, $h$ goes to 0 and $nh$ converges to $\infty$. In addition, we make the following assumptions   on $h$ and $n$:  
\begin{enumerate}
\item [(1)] When $H \in (0,\frac{3}{4})$, $nh^p \rightarrow 0$ for some $p \in (1,\frac{3+2H}{1+2H} \wedge (1+2H))$ as $n \rightarrow \infty$.
\item [(2)] When $H=\frac{3}{4}$, $\frac{nh^p}{\log(nh)} \rightarrow 0$ for some $p \in (1, \frac 95)$ as $n \rightarrow \infty$.
\item [(3)] When $H \in (\frac{3}{4},1)$, $nh^p \rightarrow 0$ for some $p \in (1, \frac{3-H}{2-H})$ as $n\to \infty$.
\end{enumerate}
Set 
	\begin{equation}
	    \bar\theta_n=\left(\frac{1}{n\sigma^2H\Gamma(2H)}\sum_{k=1}^{n}X_{kh}^2\right)^{-\frac{1}{2H}} \,.
    \end{equation}
Then $\bar\theta_n$ converges to $\theta$ almost surely as $n \to \infty$. Moreover, as $n$ tends to infinity, we have the following central and noncentral limit theorems. 
\begin{enumerate}
\item [(1)]  When  $H \in (0,\frac{3}{4})$, $\sqrt{nh} (\bar\theta_n-\theta) \xrightarrow{\mathcal{L}} N(0,\frac{\theta}{(2H)^2} \sigma_H^2)$, where $\sigma_H^2$ is given in Theorem \ref{thetaclt}. 
\item [(2)]  When  $H=\frac{3}{4}$,  $\frac{\sqrt{nh}}{\log(nh)} (\bar\theta_n-\theta) \xrightarrow{\mathcal{L}} N(0,\frac{16\theta}{9\pi})$.  
\item [(3)] When $H \in (\frac{3}{4},1)$, $(nh)^{2-2H} (\bar\theta_n-\theta) \xrightarrow{\mathcal{L}} \frac{-\theta^{2H-1}}{H \Gamma(2H+1)} R_1$, where $R_1=I_2(\delta_{0,1})$ is the Rosenblatt random variable  and $\de_{0,1}$ is the Dirac-type function defined in \eqref{ddel.def}.
\end{enumerate}

\end{thm}

Before we prove Theorem \ref{disctheta}, we state and prove an auxillary result in the following lemma about the regularity of sample paths of the fractional Ornstein-Uhlenbeck process $X$.   
\begin{lemma} \label{hd.xt}
Let $X_t$ be given by \eqref{xt.solution}. Then for every interval $[0,T]$ and any $0 < \epsilon < H$,
  \begin{equation}	
	|X_t - X_s| \leq V_1 |t-s|^{H-\epsilon} +  V_2 |t-s| \quad {\rm a.s.} ,
  \end{equation}
where the random variables  $V_i$ are defined as follows: $V_1 = \sigma \eta_T$  where $\eta_T$ is given by \eqref{hd.fbm} with $\alpha=H-\epsilon$, $V_2 =  2 \sigma \theta \sup_{u \in [0,T]} |B_u^H|$.
\end{lemma}

\begin{proof}
Consider the process $Q_t = \sigma \theta \int_0^t B_v^H e^{-\theta (t-v)} dv$. Using \eqref{t4-1-1}, for any $s,t \in [0,T]$ and $s<t$, we have
\[
  \left|X_t - X_s\right| = \left|\sigma (B_t^H - B_s^H) -  (Q_t - Q_s) \right| \leq \sigma \left|B_t^H - B_s^H\right| + \left|Q_t - Q_s\right|.
\]
Note that
\begin{eqnarray*}
  |Q_t - Q_s| & \leq & \sigma \theta \left|\int_s^t B_v^H e^{-\theta(t-v)} dv\right| + \sigma \theta \left|\int_0^s B_v^H (e^{-\theta(t-v)} - e^{-\theta(s-v)}) dv\right| \\
  & \leq & \sigma \theta \sup_{v \in [s,t]} |B_v^H| \int_s^t e^{-\theta(t-v)} dv + \sigma \theta \sup_{v \in [s,t]} |B_v^H| \left(1-e^{-\theta(t-s)}\right) \int_0^s e^{-\theta(s-v)} dv \\
  & \leq & 2 \sigma \theta \sup_{v \in [s,t]} |B_v^H| |t-s| \,.
\end{eqnarray*}
Using the above inequality for $|Q_t - Q_s|$ and Applying \eqref{hd.fbm}, with $\alpha =H-\epsilon$, for $B_t^H - B_s^H$ yield
\[
  \left|X_t - X_s\right| \leq \sigma \eta_T |t-s|^{H-\epsilon} + 2 \sigma \theta \sup_{u \in [s,t]} |B_u^H| |t-s| \,.
\]

\end{proof}

\begin{proof}[Proof of Theorem \ref{disctheta}:]   

Let $T=nh$, $Z_n = \frac{1}{nh}\int^{nh}_0 X_t^2dt$, and $\psi_n = \frac{1}{n}\sum^{n}_{k=1}X_{kh}^2$.  Consider the function 
\[ 
  f(x)=\sqrt{x}\mathbf{1}_{\{0<H<3/4\}} + \sqrt{x}/\log(x) \mathbf{1}_{\{H=3/4\}} + x^{2-2H}\mathbf{1}_{\{3/4<H<1\}} \,.
\]
{\bf Step 1:}  We claim that $f(nh)\left|Z_n - \psi_n\right| \to 0$ almost surely as $n \to \infty$. 
Applying Markov's inequality for $ \delta>0, q>1$ yields
\begin{equation}\label{cd.diff}
  P(f(nh)\left|Z_n - \psi_n\right| > \delta) \leq \delta^{-q} f(nh)^q \mathbb{E} \left|Z_n - \psi_n\right|^q \,.
\end{equation}
We apply Minkowski's inequality to obtain
\begin{eqnarray*}
  \mathbb{E} \left|Z_n - \psi_n\right|^q & = & (nh)^{-q} \mathbb{E} \Big| \sum^{n}_{j=1} \int_{(j-1)h}^{jh} (X_t+X_{jh})(X_t-X_{jh})dt \Big|^q \\
  & \leq & (nh)^{-q}  \left(\sum_{j=1}^n \int_{(j-1)h}^{jh} \left(\mathbb{E}(|X_t + X_{jh}||X_t - X_{jh}|)^q\right)^{1/q} dt \right)^q \,.
\end{eqnarray*}
Taking into account of Lemma \ref{hd.xt}, we have
\[
  \mathbb{E} \left|Z_n - \psi_n\right|^q \leq (nh)^{-q} \left( \sum_{j=1}^n \int_{(j-1)h}^{jh} \|V_1(X_t + X_{jh})\|_{L^q} |t-jh|^{H-\epsilon} + \|V_2(X_t + X_{jh})\|_{L^q} |t-jh| dt\right)^q \,,
\]
where the $V_i$'s are defined in Lemma \ref{hd.xt}. By H\"older's inequality and the fact $\|X_t\|_{L^q} = (\mathbb{E}|X_t|^q)^{1/q} \leq M_q$ for all $t>0$, $q>1$, we can write
\[
  \|V_i(X_t + X_{jh})\|_{L^q} \leq 2  M_{qr_i} \|V_i\|_{qs_i} \,,
\]
where $1/r_i + 1/s_i = 1$. Therefore,
\[
  \mathbb{E} \left|Z_n - \psi_n\right|^q \leq C \left( M_{qr_1}^q \|V_1\|_{qs_1}^q h^{q(H-\epsilon)} + M_{qr_2}^q \|V_2\|_{qs_2}^q h^{q}\right) \,,
\]
where $C$ denotes a generic constant.
 
By \eqref{hd.fbm}, $\|V_1\|_{qs_1}^q =C T^{q \epsilon}$ for $\epsilon \in (0,H)$. By the self-similarity property of fBm, $\|V_2\|_{qs_2}^q =C T^{q H}$. Using these observations, we obtain
\[
  \mathbb{E} \left|Z_n - \psi_n\right|^q \leq C \left((nh)^{q \epsilon} h^{q(H-\epsilon)} + (nh)^{q H} h^{q} \right) \,,
\]
and plugging this inequality to \eqref{cd.diff}, we get
\begin{equation}\label{cd.diff.bc}
  P(f(nh)\left|Z_n - \psi_n\right| > \delta) \leq  C \delta^{-q} f(nh)^q \left((nh)^{q \epsilon} h^{q(H-\epsilon)} + (nh)^{qH} h^{q}   \right) \,.
\end{equation}
If the right-hand side of the above inequality is summable with respect to $n$, then    $f(nh)\left|Z_n - \psi_n\right| \to 0$ almost surely by the  Borel-Cantelli Lemma. We show this summability when $H \in (0,1/2)$ and the other cases are similar. The right-hand side of \eqref{cd.diff.bc} can be written as
\[
C  n^{-1-\lambda} \left((nh^{\beta_1})^{\gamma_1} + (nh^{\beta_2})^{\gamma_2} \right) \,,
\]
where
\[
 \beta_1 = \frac{q/2 + q \epsilon+q(H-\epsilon)}{1+\lambda+q \epsilon+q/2}\,,  \quad \beta_2 = \frac{3/2q+qH}{1+\lambda+q/2+qH}  \,,
\]
and $\gamma_i$'s are the denominator of $\beta_i$'s. Note that the positive variables $\epsilon $ and $\lambda$ can be arbitrarily small and $q$ can be arbitrarily  large. In this way, we have $\beta_1 \in (1, 1+2H)$ and  $\beta_2 \in (1,\frac{3+2H}{1+2H})$. If $nh^p \to 0$ for some $p \in (1, \min(\frac{3+2H}{1+2H}, 1+2H))$, then $nh^{\beta_i} \to 0$ by carefully choosing these free variables.\\

\noindent {\bf Step 2:} We prove the almost sure convergence of $\bar\theta_n$.  Denote $\rho = \sigma^2 H \Gamma(2H)$. Recall that $\widetilde\theta_T$ is given in Theorem \ref{tthetaclt}. By  the mean value theorem, we can write 
	\begin{equation} \label{btheta.mt}
	  \bar\theta_n-\theta=\left(\frac{\psi_n - Z_n}{\rho}+\widetilde\theta_T^{-2H}\right)^{-\frac{1}{2H}}-\theta = \widetilde\theta_T - \theta + \int_0^1 g_n(\lambda) d\lambda,
	\end{equation}
where $g_n(\lambda) = -\frac{1}{2H} \frac{\psi_n - Z_n}{\rho} \left(\lambda\frac{\psi_n - Z_n}{\rho} + \widetilde\theta_T^{-2H}\right)^{-\frac{1}{2H}-1}$. \\

The result in Step 1 also implies $Z_n - \psi_n \to 0$ almost surely as $n \to \infty$, so $\lim_{n \to \infty} g_n(\lambda) = 0$ a.s. for all $\lambda \in [0,1]$. Meanwhile, for almost all $\omega$,  there exists $N:=N(\omega) \in \mathbb{N}$ such that for $n > N$, \\
\[
\left|\frac{\psi_n - Z_n}{\rho}\right| < \frac{1}{3}\theta^{-2H}, \quad \left|\widetilde\theta_T^{-2H} - \theta^{-2H}\right| < \frac{1}{3} \theta^{-2H}\,.
\]
Then for $n>N$, $|g_n(\lambda)| \leq C \theta$.  By the dominated convergence theorem,
\[
  \lim_{n \to \infty}\int_0^1 g_n(\lambda) d\lambda = 0 \quad {\rm a.s.} \,.
\]
Then it is clear that $\bar\theta_n$ converges to $\theta$ almost surely. \\

\noindent {\bf Step 3:} We prove the asymptotic laws of $\bar\theta_n$.  Equation \eqref{btheta.mt} yields
	\[
	  f(nh)(\bar\theta_n-\theta) = f(T)(\widetilde\theta_T - \theta) + f(nh)\int_0^1 g_n(\lambda) d\lambda \,.
	\]
Using the result of Step 1 and the similar arguments in step 2, we obtain
\[
  \lim_{n \to \infty} \int_0^1 f(nh) g_n(\lambda) d\lambda = 0 \quad {\rm a.s.} \,.
\]
Then it is clear that $f(nh)(\bar\theta_n-\theta)$ converges in law to the same random variable as $f(T)(\widetilde\theta_T-\theta)$  when $T$ tends to infinity.  By Theorem \ref{tthetaclt}, we finish the proof.

\end{proof}

\section{Appendix}
This section contains some technical results needed in the proofs of the main theorems of the paper.  First we need to identify the limits of some multiple integrals. 
Denote  
  \begin{equation}\label{psi.def}
	  \psi(x,u) = \psi_T(x,u)= T^{4H+1}e^{-\theta T(u+x)} \;,
  \end{equation} 
     \begin{eqnarray}
       \varphi_1 (x) & := & \int^1_x [(t-x)^{2H-1}-1][(1-t)^{2H-1} - (1-t+x)^{2H-1}]dt \,,\label{vp1.def}\\
	   \varphi_2 (x) & := & \int^1_x [(t-x)^{2H-1} - t^{2H-1}] (1-t+x)^{2H-1}dt \,,\label{vp2.def}\\
   	   \varphi_3 (x,u) & := & \int^1_x [\sgn(u-t)|u-t|^{2H-1}-\sgn(x+u-t)|x+u-t|^{2H-1}]dt \,,\label{vp3.def}\\
	   \varphi_4 (x,u) & := & \int^1_x t^{2H-1}\sgn(x+u-t)|x+u-t|^{2H-1} \nonumber \\
	    & & - (t-x)^{2H-1}\sgn(u-t)|u-t|^{2H-1} dt \,, \label{vp4.def}  \\ 
	   \varphi_5 (x,u) & := & \int_x^1 \sgn(x+u-t)|x+u-t|^{2H-1}(1-t)^{2H-1} \nonumber \\
	    & & - \sgn(u-t)|u-t|^{2H-1} (1-t+x)^{2H-1} dt \,. \label{vp5.def}\\ \nonumber
     \end{eqnarray}
Fix an $\epsilon\in (0,  \frac{1}{4})$. Denote $[0,1]^2 = \mathcal{I}_1 \cup \mathcal{I}_2$ where $\mathcal{I}_1 = [0,\epsilon]^2$ and $\mathcal{I}_2 = [0,1]^2 \backslash [0,\epsilon]^2$. 
\begin{lemma}
  Let $H \in (0, \frac{1}{2})$. When $x, u\in \mathcal{I}_1$, we have the following estimates.
\begin{enumerate}
\item[(i)] 
\begin{equation}
 |\varphi_1(x) | \le x^{2H} \,,  \label{ecua1}
 \end{equation}
\item [(ii)]
 \begin{equation}
 |\varphi_3(x,u) | \le C(x^{2H} + u^{2H} + |u-x|^{2H}) \,, \label{ecua2}
 \end{equation}
\item[(iii)]
 \begin{equation}
 |\varphi_5(x,u) | \le C(x^{2H} + u^{2H} + |u-x|^{2H}) \,, \label{ecua3}
 \end{equation} 
 \end{enumerate} 
 where $C$ is a constant independent of $x,u$.
 \end{lemma}
 
 \begin{proof}  First we prove  \eqref{ecua1}. 
 Observe that 
     \begin{equation} \label{vp1.i1}
		 0 \leq \varphi_1(x) \leq \int_x^1 f(x,t) dt,
	 \end{equation} 
where 
      \[
      f(x,t)=(t-x)^{2H-1}[(1-t)^{2H-1}-(1-t+x)^{2H-1}]\,.
      \]
It is clear that  for $\frac{1+x}{2} \leq t \leq 1$
  \[
    f(x,t) \leq \big(\frac{1-x}{2}\big)^{2H-1} [(1-t)^{2H-1}-(1-t+x)^{2H-1}] \,,
  \]
whereas for $x \leq t \leq \frac{1+x}{2}$.  Applying the mean value theorem for the second factor of $f(x,t)$ yields
  \[
    f(x,t) \leq (1-2H)(t-x)^{2H-1}\big(\frac{1-x}{2}\big)^{2H-2} x \,.
  \]
Integrating the right-hand side of the above two inequalities with respect to $t$, we obtain 
	  \begin{eqnarray*}
	  \int_x^1 f(x,t)dt & \leq & \frac{1}{2H} \big(\frac{1-x}{2}\big)^{2H-1} \big[\big(\frac{1-x}{2}\big)^{2H} - \big(\frac{1+x}{2}\big)^{2H} + x^{2H}\big] + \frac{1-2H}{2H} \big(\frac{1-x}{2}\big)^{4H-2}x \\
	  & \leq & \frac{1}{2H} \big(\frac{1-\epsilon}{2}\big)^{2H-1} x^{2H} + \frac{1-2H}{2H} \big(\frac{1-\epsilon}{2}\big)^{4H-2}x^{2H} \;,
	  \end{eqnarray*}
where we have used the inequality $x<x^{2H}$ on $\mathcal{I}_1$ (i.e., $x \in (0,\epsilon)$). Thus, \eqref{ecua1} follows from the above inequality and \eqref{vp1.i1}.\\
 
  Next we prove \eqref{ecua2}.  
Note that the antiderivative of the function $\sgn(x)|x|^{2H-1}$ is $(2H)^{-1} |x|^{2H}$,  so we can compute $\varphi_3 (x,u)$ as follows
	  	\begin{equation} \label{vp3.ddef}
	  	 \varphi_3 (x,u) = \frac{1}{2H}( |u-x|^{2H} - (1-u)^{2H} + (1-x-u)^{2H} - u^{2H})\,. \\
	  	\end{equation}
 Applying the inequality 
	 \[
	 \left|(1-x-u)^{2H} - (1-u)^{2H}\right| \leq 2H (1-x-u)^{2H-1} x \leq 2H (1-2\epsilon)^{2H-1} x^{2H} \;,
	 \] 
	 and the triangular inequality to \eqref{vp3.ddef}  yields   
	 		\begin{eqnarray*}
	 		  |\varphi_3 (x,u)| &\leq& (2H)^{-1}(|u-x|^{2H} + u^{2H} + 2H(1-2\epsilon)^{2H-1}x^{2H}) \\
			  & \leq & C(|u-x|^{2H} + u^{2H} + x^{2H})   
  \quad \forall \;  x,u \in \mathcal{I}_1 \,.
	 		\end{eqnarray*} 
Finally,  we prove \eqref{ecua3}. Denote $$\zeta_{x,u}(t) = \sgn(x+u-t)|x+u-t|^{2H-1}(1-t)^{2H-1} - \sgn(u-t)|u-t|^{2H-1} (1-t+x)^{2H-1} .$$
Let $\delta \in (\frac{1}{2},1)$. Since $\epsilon \in (0,\frac{1}{4})$ and $(x,u) \in (0,\epsilon)^2$,  the interval $(x, 1)$ can be decomposed into the following three intervals, where \[
    J_1 = (x, \, u+x), \quad J_2 = (u+x, \, \delta), \quad J_3 = (\delta, \, 1) \,.
  \]
 Then $\varphi_5(x,u) = \sum_{k=1}^3 \int_{J_k} \zeta_{x,u}(t) dt$. We consider the above three integrals separately. \\
 
 \noindent
{\it Case 1:} When $ t \in J_1$, we have 
    \begin{equation} \label{j1.a}
	   (1-t)^{2H-1} \leq (1-u-x)^{2H-1} \leq (1 - 2\epsilon)^{2H-1} \,.
    \end{equation}
When $t$ falls in different subintervals of $J_1$, we bound $(1-t+x)^{2H-1}$ in different ways. Namely, if $t \in (x,u)$ and $ u \geq x$,
    \begin{equation} \label{j1.b}
      (1-t+x)^{2H-1} \leq (1+x-u)^{2H-1} \leq (1 - \epsilon)^{2H-1} \,.
    \end{equation}
If $t \in (x\vee u, x+u)$,
    \begin{equation} \label{j1.c}
      (1-t+x)^{2H-1} \leq (1-u)^{2H-1} \leq (1 - \epsilon)^{2H-1} \,.
    \end{equation}  
Applying \eqref{j1.a} for the first summand in $\zeta_{x,u}(t)$, \eqref{j1.b} and \eqref{j1.c} for the second summand, we can bound the integration of $\zeta_{x,u}(t)$ on $J_1$ as follows
        \begin{eqnarray*}
          \Big| \int_{J_1}\zeta_{x,u}(t)dt \Big| & \leq & (1-2\epsilon)^{2H-1} \int_{x}^{u+x} (x+u-t)^{2H-1} dt \\
		  & & + \, (1-\epsilon)^{2H-1} \big ( \int_x^u (u-t)^{2H-1} 1_{\{u \geq x\}} dt + \int_{x \vee u}^{x+u} (t-u)^{2H-1} dt \big ) \,.
        \end{eqnarray*}
	Integrating with respect to $t$  yields
	\[
	   \Big| \int_{J_1}\zeta_{x,u}(t)dt \Big| \leq C (u^{2H}  + (u-x)^{2H} 1_{\{u \geq x\}} + x^{2H}) \,.
	\]

\noindent		
{\it Case  2:} For $t \in J_2$, we rewrite 
    \begin{eqnarray*}
	   - \int_{J_2} \zeta_{x,u}(t) dt & = & \int_{u+x}^{\delta} (1-t)^{2H-1} \big( (t-u-x)^{2H-1}-(t-u)^{2H-1} \big) \\
		&& + \; (t-u)^{2H-1} \big( (1-t)^{2H-1}-(1-t+x)^{2H-1} \big) dt \,,
    \end{eqnarray*}
	which is nonnegative. In the above integrand, we bound $(1-t)^{2H-1}$ by $(1-\delta)^{2H-1}$ for the first summand. For the second summand, we apply  the mean value theorem for the difference part and bound $(t-u)^{2H-1}$ by $x^{2H-1}$.  Then integrating $t$ yields
	\begin{eqnarray*}
	  0 \leq - \int_{J_2} \zeta_{x,u}(t) dt & \leq & \frac{(1-\delta)^{2H-1}}{2H} \big((\delta - u - x)^{2H} - (\delta - u)^{2H} + x^{2H} \big) \\
	   & & + \; x^{2H} \big ((1-\delta)^{2H-1} - (1-u-x)^{2H-1} \big ) \\
	& \leq & \frac{(1-\delta)^{2H-1}}{2H} x^{2H} + (1-\delta)^{2H-1} x^{2H} \leq C x^{2H}\;.
	\end{eqnarray*}
	
\noindent
{\it Case 3:} For $t \in J_3$, we rewrite
   \begin{eqnarray*}
	  - \int_{J_3} \zeta_{x,u}(t) dt & = & \int_{\delta}^1 (t-u-x)^{2H-1}((1-t)^{2H-1}-(1-t+x)^{2H-1}) \\
	&& + \; (1-t+x)^{2H-1}((t-u-x)^{2H-1}-(t-u)^{2H-1}) dt \,,
   \end{eqnarray*}
which is nonnegative. In the above integrand, we bound $(t-u-x)^{2H-1}$ by $(\delta - 2\epsilon)^{2H-1}$ for the first summand. For the second summand, apply the mean value theorem for the difference part and bound $(1-t+x)^{2H-1}$ by $x^{2H-1}$. Then integrating $t$ yields
\begin{eqnarray*}
   0 \leq - \int_{J_3} \zeta_{x,u}(t) dt & \leq & \frac{(\delta - 2\epsilon)^{2H-1}}{2H} \big( (1 - \delta)^{2H} - (1 -  \delta +x)^{2H} + x^{2H} \big) \\
& & + \; x^{2H} \big( (\delta - u - x)^{2H-1}-(1-u-x)^{2H-1} \big) \\
& \leq & \frac{(\delta - 2\epsilon)^{2H-1}}{2H} x^{2H} +  x^{2H} (\delta - u - x)^{2H-1}  \leq  C x^{2H} \,.
\end{eqnarray*}
In the last step  we have applied the inequality $\delta - u - x \geq \delta - 2\epsilon$.
\end{proof}

\begin{lem} \label{psi.phi4}
Suppose $H \in (0, \frac{1}{2})$.  Let $\psi(x,u)$ and $\varphi_4(x,u)$ defined by \eqref{psi.def} and \eqref{vp4.def}, respectively. Fix $\epsilon \in (0,1/4)$. Then
 \begin{equation}\label{pp4.eq1}
 \lim_{T \to \infty} \int_{[0,\epsilon]^2} \psi(x,u) (x^{2H} + u^{2H} + |x-u|^{2H}) = 0 \,,
 \end{equation}
and
\begin{eqnarray}
 && \lim_{T \to \infty} \int_{[0,1]^2} \psi(x,u)\varphi_4(x,u) dxdu  \nonumber \\
       &  & \qquad = \theta^{-1-4H} \Big(\Gamma(2H)^2 (2H-2^{-1})  + \frac{\Gamma(2-4H)\Gamma(2H)\Gamma(4H)}{\Gamma(1-2H)}\Big) \label{pp4.eq2} \,.
\end{eqnarray}
\end{lem}

\begin{proof} 
    We first prove \eqref{pp4.eq1}. For the first summand, making the change of variables $Tx \to x_1$ and $Tu \to x_2$ yields
    	     \begin{equation}\label{psi.x}	 
    		    \int_{[0,\epsilon]^2} T^{4H+1}e^{-\theta T(x+u)} x^{2H} dxdu = T^{2H-1} \int_{[0,T\epsilon]^2} e^{-\theta(x_1 + x_2)} x_1^{2H} dx_1 dx_2 \;,
    		 \end{equation}
    	which goes to $0$ as $T \to \infty$. A similar argument could be applied to the second summand. For the third summand, by symmetry it suffices to consider the integral on  the region $\{u>x\}$. Making the change of variables $T(u-x) \to x_1$, $Tx \to x_2$ yields
    	     \begin{equation}\label{psi.xu}	 
    		   \int_{[0,\epsilon]^2} T^{4H+1}e^{-\theta T(x+u)} |u-x|^{2H} dxdu = 2 T^{2H-1} \int_{[0,T\epsilon]^2, x_1 + x_2 \leq T \epsilon} e^{-\theta(x_1 + 2x_2)} x_1^{2H} dx_1 dx_2 \;,
    		 \end{equation}
    which goes to $0$ as $T \to \infty$.\\
 
 Next we show \eqref{pp4.eq2}. Set 
	\[
	\Theta:= \lim_{T \rightarrow \infty} \int_{[0,1]^2} \psi(x,u)\varphi_4(x,u) dxdu .
	\]
Making change of variables, $\theta T x \to x, \theta T u \to u, \theta T t \to t$, we can write
	\begin{eqnarray*}
 \Theta &=&   \theta^{-1-4H}\int_{[0,\infty)^2}e^{-(u+x)}dxdu \\
 &&\times \int^{\infty}_x [t^{2H-1}\sgn(x+u-t)|x+u-t|^{2H-1} - (t-x)^{2H-1}\sgn(u-t)|u-t|^{2H-1}]dt \,.
    \end{eqnarray*}
   The above integral can be decomposed as follows \[\Theta = \theta^{-1-4H} (L_1 - L_2 + L_3) \;,\] where
   \[
    L_1  :=  \int_{[0,\infty)^2} e^{-(x+u)} dxdu \int_{x}^{x+u} t^{2H-1}(x+u-t)^{2H-1} dt \;,
   \]
   \[
   L_2  :=  \int_{[0,\infty)^2,u>x} e^{-(x+u)} dxdu \int_{x}^{u} (t-x)^{2H-1}(u-t)^{2H-1} dt  \;,
   \]
   \[
   L_3 := \int_{[0,\infty)^2} e^{-(x+u)} dxdu \Big(\int_{u \vee x}^{\infty} (t-x)^{2H-1}(t-u)^{2H-1} dt - \int_{x+u}^{\infty} t^{2H-1}(t-x-u)^{2H-1}dt\Big) \,.
   \]
   Making the change of variables $t-x \to s$ and integrating $u$, we obtain
     \[ 
       L_1 = \Gamma(2H) \int_{[0,\infty)^2} e^{-(x+s)} (x+s)^{2H-1} dxds = \Gamma(2H)^2 2H \,.
     \]
   Denote by ${\rm B} (\alpha, \beta) $ the Beta function. Then
    \[
      L_2  =  {\rm B} (2H, 2H) \int_{[0,\infty)^2,u>x} e^{-(x+u)} (u-x)^{4H-1} dxdu  \,.
     \]
   By setting $u-x \to v$ and integrating in $x$ first, we deduce $L_2 = \Gamma(2H)^2/2$.
 To compute $L_3$, by symmetry it suffices to integrate on the region $\{u<x\}$.  For the second integral, we make the change of variables $t-u \to y $. In this way, we obtain
    \[
      L_3 = 2 \int_{  0<u<x<y<\infty } e^{-(u+x)}((y-u)^{2H-1}-(y+u)^{2H-1})(y-x)^{2H-1}dydxdu.
    \]
The change of variables $x-u \to a, y-x \to b$ yields
    \begin{eqnarray*}
	  L_3 & = & 2 \int_{\mathbb{R}_+^3} e^{-(a+2u)} b^{2H-1} \big[(a+b)^{2H-1} - (a+b+2u)^{2H-1}\big]dudadb \\
	  & = & 2 \int_{\mathbb{R}_+^3} e^{-(a+2u)} b^{2H-1} dudadb \int_a^{2u+a} (1-2H)(b+z)^{2H-2} dz \\
	  & = & 2(1-2H) \int_{\mathbb{R}_+^2} e^{-(a+2u)} duda \int_a^{2u+a} (\int_{\mathbb{R}_+} b^{2H-1} (b+z)^{2H-2} db) dz \,.
	\end{eqnarray*}
 Setting $z/(b+z) \to v$ and integrating $v$ on $[0,1]$, we obtain
    \begin{eqnarray*}
	  L_3 & = & 2(1-2H) {\rm B }(2-4H, 2H) \int_{\mathbb{R}_+^2} e^{-(a+2u)} duda \int_a^{2u+a} z^{4H-2}  dz \\
	  & = & \frac{\Gamma(2-4H)\Gamma(2H)\Gamma(4H)}{\Gamma(1-2H)} \,.
	\end{eqnarray*}
Then, the lemma follows from the above computations  of $L_1$, $L_2$ and $ L_3$.
   \end{proof}

\begin{lemma}
 Denote $\mathcal{I}_1 = [0,\epsilon]^2$ and $\mathcal{I}_2 = [0,1]^2 \backslash [0,\epsilon]^2$.  The functions $\psi$ and $\varphi_i$ are given by \eqref{psi.def} to \eqref{vp5.def}. For $j=1, 2$ and $i=1,2,3,5$, we have the following result.
    \begin{equation} \label{vp1}
   	  \lim_{T\rightarrow \infty} \int_{\mathcal{I}_j} \psi \varphi_i dxdu = 0  \,.
   	 \end{equation}
\end{lemma}	

\begin{proof}
The proof of \eqref{vp1} is divided into the cases $j=2$ and $j=1$.\\

\noindent  {\it{Case  $j=2$}:}  Clearly, for $(x,u) \in \mathcal{I}_2$,
	 \begin{equation}\label{psi}
	  \psi(x,u) \leq T^{4H+1} e^{-\theta T\epsilon},
	 \end{equation}
which implies
	 \begin{equation}\label{vp4}
	   \int_{\mathcal{I}_2} \psi \varphi_i dxdu \;  \to 0 {\text{\quad for} \;} i=1, 2, 3, 5  
	 \end{equation}
as $T \to \infty$. Thus, \eqref{vp1} holds true for $j=2$.  \\

\noindent {\it{Case $j=1$}:}   For $i=2$, we evaluate the integral of $\psi \varphi_2$ on $\mathcal{I}_1$ by making change of variables $Tx \to x$, $Tu \to u$ and $Tt \to t$. In this way, we obtain
 	  \begin{equation}
 	   \int_{\mathcal{I}_1}\psi \varphi_2 dxdu = \int_{[0,T\epsilon]^2}e^{-\theta (u+x)}dxdu \int^T_x[(t-x)^{2H-1}-t^{2H-1}](T-t+x)^{2H-1}dt \,.
 	  \end{equation}
 Clearly $(T-t+x)^{2H-1} \leq x^{2H-1}$, so the integrand of the above triple integral is bounded by the function $e^{-\theta (u+x)} ((t-x)^{2H-1}-t^{2H-1}) \mathbf{1}_{\{t \geq x\}} x^{2H-1}$ which is integrable on $[0,\infty)^3$.  
 As $T \to \infty$, $(T-t+x)^{2H-1} \to 0$.  Applying the dominated convergence theorem, we have
 	\begin{equation} \label{vp2}
 		\lim_{T \to \infty}\int_{\mathcal{I}_1} \psi \varphi_2 dxdu = 0.
 	\end{equation}
The cases $i=1, 3, 5$ follows from  \eqref{ecua1}, \eqref{ecua2} and \eqref{ecua3} and Lemma \ref{psi.phi4}.
\end{proof}

\begin{lem} \label{fxi}
  Let $H \in (0, \frac{1}{2})$. Denote $f_{\alpha}(\xi,\eta,\eta') = |\xi-\eta|^{-1+\alpha} |\xi-\eta'|^{-1+\alpha} \xi^{1-2H}$ where $\xi,\eta,\eta' \in \mathbb{R}_+$. Assuming $\eta \geq \eta'$ and $0< \alpha < \frac{1}{2}$, we have the following results. 
  
  \begin{description}
   \item [(i)] For any $0<\epsilon<\alpha$, there exists some positive constants $K_1, K_2$ depending on $\alpha, \epsilon$ such that
  \begin{equation*}
	\int_{(0,\eta+\eta')}f_{\alpha}(\xi,\eta,\eta') d\xi \leq K_1 (\eta-\eta')^{-1+2\alpha-\epsilon} \eta^{1-2H+\epsilon} + K_2 (\eta-\eta')^{-1+2\alpha} \eta^{1-2H} \,.
  \end{equation*}
  
  \item [(ii)] If $\alpha \in (0,H)$, then there exists some positive constants $K_3, K_4$ depending on $\alpha$ and $H$ such that
  \begin{eqnarray*}
	\int_{[\eta+\eta', \infty)} f_{\alpha}(\xi,\eta,\eta') d\xi \leq K_3 (\eta')^{2\alpha-2H} + K_4 \eta^{1-2H} (\eta')^{-1+2\alpha} \,.
  \end{eqnarray*}
  \end{description}
\end{lem}

\begin{proof}
   We  partition $(0,\eta+\eta')$ into three intervals:
  $(0, \eta'] \cup (\eta', \eta] \cup (\eta, \eta+\eta')$. We shall use the inequality
  \begin{equation} \label{eineq}
	  (a+b)^{-r} \leq a^{-s}b^{-r+s},\quad  \hbox{ for any} \ 0 < s < r {\rm \; and \;} a, b \in \mathbb{R}_+.
  \end{equation}
For any $0 < \epsilon < \alpha$, we write $\xi - \eta' = (\eta - \eta') + (\xi - \eta)$ and apply \eqref{eineq} for $(\xi - \eta')^{-1+\alpha}$ with $s=1-2\alpha+\epsilon$. In this way, we obtain 
 \begin{eqnarray}
  \int_{\eta}^{\eta+\eta'} f_{\alpha}(\xi,\eta,\eta') d\xi & \leq & \int_{\eta}^{\eta+\eta'} (\xi-\eta)^{-1+\alpha} (\eta-\eta')^{-(1-2\alpha+\epsilon)} (\xi-\eta)^{(-1+\alpha)+(1-2\alpha+\epsilon)} \xi^{1-2H} d \xi \nonumber \\
  &\leq &  (\eta - \eta')^{-1+2\alpha-\epsilon} (\eta + \eta')^{1-2H} \int_{\eta}^{\eta+\eta'} (\xi-\eta)^{-1 + \epsilon} d\xi \nonumber \\
  &\leq& 2^{1-2H}\epsilon^{-1} (\eta - \eta')^{-1+2\alpha-\epsilon} \eta^{1-2H+\epsilon} \label{fxi-2} \;.
 \end{eqnarray}
For $\xi \in (\eta', \eta]$, observe that
\begin{eqnarray}
  \int_{\eta'}^{\eta} f_{\alpha}(\xi,\eta,\eta') d \xi & \leq & \eta^{1-2H} \int_{\eta'}^{\eta} (\eta-\xi)^{-1+\alpha} (\xi-\eta')^{-1+\alpha} d\xi \nonumber \\
  & \leq & \eta^{1-2H} \int_0^{\eta - \eta'} ( \eta - \eta' - \xi)^{-1+\alpha} \xi^{-1+\alpha} d \xi \nonumber \\
  & \leq &  K_2 \eta^{1-2H} (\eta - \eta')^{-1+2\alpha} , \label{fxi-3}
\end{eqnarray}
where $ K_2 = {\rm B}(\alpha, \alpha) $.  
For $\xi \in (0, \eta']$, writing $\eta - \xi = (\eta - \eta') + (\eta' - \xi)$ and applying (\ref{eineq}) with $s=1-2\alpha+\epsilon$ again (for the same $\epsilon$ as above), we obtain
\begin{eqnarray}
  \int_0^{\eta'} f_{\alpha}(\xi,\eta,\eta') d \xi & \leq & (\eta')^{1-2H} \int_0^{\eta'} (\eta-\xi)^{-1+\alpha} (\eta' - \xi)^{-1+\alpha} d\xi\\
  & \leq & (\eta')^{1-2H} \int_0^{\eta'} (\eta - \eta')^{-(1-2\alpha+\epsilon)} (\eta' - \xi)^{(-1+\alpha)+(1-2\alpha+\epsilon)} (\eta' - \xi)^{-1+\alpha}   d\xi\nonumber \\
  & \leq & \epsilon^{-1} \eta^{1-2H+\epsilon} (\eta - \eta')^{-1+2\alpha-\epsilon} \label{fxi-4}.
\end{eqnarray}
Let $K_1 = 3\epsilon^{-1}$.  By  \eqref{fxi-2},  \eqref{fxi-3},  and \eqref{fxi-4},  the first part of lemma is obtained.\\

Now if $\alpha \in (0,H)$, then
  \begin{eqnarray}
   \int_{\eta+\eta'}^{\infty} f_{\alpha}(\xi,\eta,\eta') d\xi & \leq & \int_{\eta+\eta'}^{\infty} (\xi - \eta)^{-2+2\alpha} \xi^{1-2H} d\xi = \int_{\eta'}^{\infty} \xi^{-2+2\alpha} (\xi + \eta)^{1-2H} d\xi \nonumber \\
   & & \leq  \int_{\eta'}^{\infty} \xi^{-2+2\alpha} (\xi^{1-2H} + \eta^{1-2H}) d \xi \nonumber \\
   & & = K_3 (\eta')^{2\alpha-2H} + K_4 \eta^{1-2H} (\eta')^{-1+2\alpha} , \label{fxi-1}
   \end{eqnarray}
   where in the third step, we apply the inequality $(\xi + \eta)^{1-2H} \leq \xi^{1-2H} + \eta^{1-2H}$ for $H \in (0,\frac{1}{2})$.  Here the constants $K_3 = (2H-2\alpha)^{-1}$, $K_4 = (1-2\alpha)^{-1}$. This finishes the proof of the lemma.\\
\end{proof}  
  
\begin{lem} \label{at}
  For $n \geq 0$, and $H \in [\frac{3}{4},1)$, set
   $$A_{1,H}(T)={T^{3-4H}}\int^T_0 \int^{T-t}_0 s^n e^{-\theta s}t^{2H-2}(s+t)^{2H-2}dsdt,$$ and
   $$A_{2,H}(T)=T^{3-4H}\int^T_0 \int^T_0 s^n e^{-\theta s}t^{2H-2}(s+t)^{2H-2}dsdt.$$
   Then
   \begin{enumerate}
    \item [(i)] For $H \in (\frac{3}{4},1)$, $\displaystyle\lim_{T \to \infty} A_{1,H}(T) = \displaystyle\lim_{T \to \infty} A_{2,H}(T) = \frac{\theta^{-(n+1)}\Gamma(n+1)}{4H-3}$;\\
    \item [(ii)] For $H=\frac{3}{4}$, $\displaystyle\lim_{T \to \infty}\frac{A_{1,H}(T)}{\log T} = \displaystyle\lim_{T \to \infty}\frac{A_{2,H}(T)}{\log T} = \Gamma(n+1)\theta ^{-(n+1)}$.\\
   \end{enumerate}
\end{lem}

\begin{proof}

 (i) For $H \in (\frac{3}{4},1)$, we have  $$A_{2,H}(T) \leq T^{3-4H} \int^T_0 \int^T_0 s^n e^{-\theta s} t^{4H-4} dsdt,$$ and $$A_{1,H}(T) \geq T^{3-4H} \int^T_0 \int^{T-t}_0 s^n e^{-\theta s} (s+t)^{4H-4} dsdt.$$  For the right-hand sides of the above two inequalities, we integrate  first in $t$  to obtain 
\begin{eqnarray*}
&&  \frac{1}{4H-3} \left(\int_0^T s^n e^{-\theta s}ds - T^{3-4H} \int_0^T s^{n+4H-3} e^{-\theta s} ds \right)\\
&&\qquad\qquad  \leq A_{1,H}(T) \leq A_{2,H}(T) \leq \frac{1}{4H-3}\int^T_0 s^n e^{-\theta s}ds.
\end{eqnarray*}
 This yields (i) by letting $T \to \infty$. \\

\noindent  (ii) For $H=\frac{3}{4}$, by the  L'Hopital rule,  we have 
 \begin{equation*}
 \lim_{T \to \infty}\frac{A_{2,H}(T)}{\log T} = \lim_{T \to \infty} T\Big[\int^T_0 s^n e^{-\theta s} T^{-\frac{1}{2}}(s+T)^{-\frac{1}{2}}ds + \int^T_0 T^n e^{-\theta T} t^{-\frac{1}{2}} (T+t)^{-\frac{1}{2}}dt\Big].
 \end{equation*}
The second summand on the right-hand side of the above equation goes to 0 as $T \to\infty$, so 
 \begin{equation}\label{ah-1}
   \lim_{T \to \infty}\frac{A_{2,H}(T)}{\log T} \leq \int^{\infty}_0 s^n e^{-\theta s}ds.
 \end{equation}\\
 On the other hand, by the inequality $t \leq s+t$, \\
 \begin{eqnarray*}
 \frac{A_{1,H}(T)}{\log T} & \geq & \frac{1}{\log T}\displaystyle \int^T_0 \int^{T-t}_0 s^n e^{-\theta s} (s+t)^{-1} dsdt \\
 & = &\frac{1}{\log T} \Big[ \log T \int^T_0 s^n e^{-\theta s}   ds - \int^T_0 s^n e^{-\theta s} \log s  ds\Big]. \\
 \end{eqnarray*}
 The function $s^n e^{-\theta s} \log s$ is integrable on $[0,\infty)$. Thus,
 \begin{equation}\label{ah-2}
	 \lim_{T\to\infty}\frac{A_{1,H}(T)}{\log T} \geq \int^{\infty}_0 s^n e^{-\theta s}ds.
 \end{equation}\\
 By (\ref{ah-1}) and (\ref{ah-2}), we conclude the proof of (ii). 
\end{proof}

\begin{lem} \label{lem5.1}  Let $F_T$, $\widetilde F_T$ be defined by \eqref{e.ft1} and \eqref{e.ft2}, respectively. Moreover, let $R_1$ be defined in Part (iii) of Theorem \ref{thetaclt}. Then we have the following convergence results.  

\begin{description}
\item[(i)]
 When  $0 < H < \frac{1}{2}$ we have
\begin{equation}
 \lim_{T\to\infty}\EE\left( \frac{1}{T}F_T^2\right) = 4H^2 \theta^{1-4H} \Gamma(2H)^2 \Big((4H-1) + \frac{2\Gamma(2-4H)\Gamma(4H)}{\Gamma(2H)\Gamma(1-2H)}\Big)
 \,. \label{e.lemm5.1}
\end{equation}
\item[(ii)] When $H = \frac{3}{4}$, we have
\begin{equation}
 \lim_{T\to\infty}\frac{\EE\left(F_T^2\right)}{T\log(T)} = 9/4 \theta^{-2} \,
 \,. \label{e.lemm5.1-2}
\end{equation}
\item[(iii)] When $H > \frac{3}{4}$,  we have 
\begin{equation}
 \lim_{T\to\infty}\EE\left(T^{2-4H}F_T^2\right) = \frac{16\alpha_H^2 \theta^{-2}}{(4H-2)(4H-3)} \,, \label{e.lemm5.1-3}
 \end{equation}
\begin{equation}
 \lim_{T \to \infty} \mathbb{E}[ T^{1-2H} R_1 \widetilde F_T]=\frac{8 \alpha_H^2 \theta^{-1}}{(4H-2)(4H-3)} \,, 
\label{e.lemm5.1-4} 
 \end{equation}
 where $\alpha_H = H(2H-1)$.
\end{description}
\end{lem}
In the above lemma, we do not give a statement when $H \in [\frac{1}{2}, \frac{3}{4})$, because this case has been studied in \cite{hn}.\\

\begin{proof}  
	{\bf Part (i):} Assume $H \in (0,\frac{1}{2})$. Applying L'Hopital's rule to \eqref{ET2} yields
    \begin{equation} \label{eft2}
		\lim_{T \to \infty} \mathbb{E}\left(\frac{1}{T} F_T^2\right) = \lim_{T \to \infty}4H^2\theta^2(I_1+I_2)\,, 
	\end{equation}
where
    \begin{eqnarray}
  & &  I_1 = (H\theta)^{-1} \int_{[0,T]^3} e^{-\theta(T-t_1)}\frac{\partial e^{-\theta|s_2-t_2|}}{\partial s_2}[T^{2H-1}-(T-s_2)^{2H-1}]\frac{\partial R_H(t_1,t_2)}{\partial t_2}ds_2dt_1dt_2 \,, \nonumber\\
 &&I_2 = - (H\theta)^{-1} \int_{[0,T]^3} e^{-\theta(T-s_1)}\frac{\partial e^{-\theta|s_2-t_2|}}{\partial s_2}\frac{\partial R_H(s_1,s_2)}{\partial s_1} \nonumber\\
 &&  \qquad \qquad  \times[t_2^{2H-1}+(T-t_2)^{2H-1}]ds_1ds_2dt_2\,. \label{eq4}
    \end{eqnarray}
To compute the limit of $\mathbb{E} (\frac{1}{T} F_T^2)$ we will consider that   of $I_1$ and $I_2$.  \\

\noindent {\it Computation of $\lim_{T\to \infty}I_1$}: \
 We first compute explicitly the partial derivatives in the integrand of $I_1$.  On the region $\{t_2 > s_2\}$, we make change of variables $1-\frac{t_1}{T} \to u$, $ \frac{t_2}{T} - \frac{s_2}{T} \to x$ and $ 1-\frac{s_2}{T} \to t$, and on the region $\{t_2 < s_2\}$, we make change of variables $1-\frac{t_1}{T} \to u$, $ \frac{s_2}{T} - \frac{t_2}{T} \to x$  and $ 1-\frac{t_2}{T} \to t$. In this way, $I_1$ can be written as
 \begin{eqnarray} 
	 I_1 & = & \int_{[0,1]^3, x \leq t} T^{4H+1} e^{-\theta T (u+x)} (1-t^{2H-1})\nonumber\\
	 &&\qquad  \left((1-t+x)^{2H-1} - \sgn(x+u-t)|x+u-t|^{2H-1}\right) dudxdt \nonumber \\
	 && - \int_{[0,1]^3, x \leq t} T^{4H+1} e^{-\theta T (u+x)} \left(1-(t-x)^{2H-1}\right) \left((1-t)^{2H-1} - \sgn(u-t)|u-t|^{2H-1}\right) dudxdt \,. \nonumber\\\label{i1.s}
 \end{eqnarray} 

Reorganize the  terms in the above integrals we have 
     \begin{equation} \label{i1}
     I_1 = \int_{[0,1]^2} \psi(x,u) \sum_{i=1}^4\varphi_i dxdu \,,
     \end{equation}
where the functions $\psi$, $\varphi_i$ are given by \eqref{psi.def} to \eqref{vp4.def}.

By \eqref{vp1}, we see 
    \begin{equation} \label{li1}
    \lim_{T\to\infty}I_1 = \lim_{T \to \infty} \int_{[0,1]^2} \psi(x,u) \varphi_4 (x,u) dxdu \,,
	 \end{equation}
whose value is computed in \eqref{pp4.eq2} of Lemma \ref{psi.phi4}.\\

\noindent {\it Computation of $\lim_{T\to\infty} I_2$:} We first compute explicitly the partial derivatives in the integrand of \eqref{eq4}.  On the region $\{s_2 > t_2\}$, we make change of variables $T - s_1 \to Tu$, $s_2 - t_2 \to Tx$ and $T - t_2 \to Tt$, and on the region $\{t_2 > s_2\}$, we make change of variables $T - s_1 \to Tu$, $t_2 - s_2 \to Tx$  and $T - s_2 \to Tt$. In this way,
\begin{eqnarray}
  I_2 & = & \int_{[0,1]^3, t \geq x} T^{4H+1} e^{-\theta T (u+x)} \left((1-u)^{2H-1} + \sgn (x+u-t)|x+u-t|^{2H-1}\right)  \nonumber\\
  & & \left(t^{2H-1} + (1-t)^{2H-1}\right)dudxdt - \int_{[0,1]^3, t \geq x} T^{4H+1} e^{-\theta T (u+x)}  \nonumber \\
  & & \left((1-u)^{2H-1} + \sgn (u-t)|u-t|^{2H-1}\right) \left((1-t+x)^{2H-1} + (t-x)^{2H-1}\right) dudxdt \,. \nonumber \\ \label{i2.s}
\end{eqnarray} 
Note that 
\[
  \int_x^1 \left(t^{2H-1} + (1-t)^{2H-1}\right) - \left((1-t+x)^{2H-1} + (t-x)^{2H-1}\right) dt = 0 \;,
\]
so $I_2$ can be simplified and rewritten as
\begin{equation}\label{i2.d}
  I_2 = \int_{[0,1]^2} \psi(x,u) \big(\varphi_4(x,u) + \varphi_5(x,u) \big) dxdu \;,
\end{equation}
where $\psi(x,u)$, $\varphi_4(x,u)$ and $\varphi_5(x,u)$  are given by \eqref{psi.def}, \eqref{vp4.def} and \eqref{vp5.def} respectively.
By \eqref{li1} and the result of \eqref{vp1} for $i=5$, we have
\begin{equation} \label{li2}
	\lim_{T \to \infty} I_2 = \lim_{T \to \infty} I_1 \,.
\end{equation}
Then part {\bf (i)} follows from \eqref{eft2},  \eqref{li1}, \eqref{li2} and \eqref{pp4.eq2}.
 \\

\noindent {\bf Part (ii) and (iii): } Assume $H \geq 3/4$. Using \eqref{ip.ghalf}, we  have 
\begin{equation} \label{eft2-2}
\mathbb{E} (F_T^2) =  2\alpha_H^2 I_T\,,
\end{equation}
where $\alpha_H = H(2H-1)$, and 
\begin{equation} \label{eq1}
 I_T=\int_{[0,T]^4}e^{-\theta |s_2-u_2|- \theta|s_1-u_1|}|s_2-s_1|^{2H-2}|u_2-u_1|^{2H-2}du_1du_2ds_1ds_2\,. 
\end{equation}
Applying   L'Hopital rule yields 
\begin{empheq}[left=\empheqlbrace]{align}
 & \lim_{T \to \infty} \mathbb{E} (T^{2-4H} F_T^2)   =   \frac{8\alpha_H^2}{4H-2} \lim_{T \to \infty} T^{3-4H} J_T  
&\qquad\qquad \hbox{when $H \in (\frac{3}{4},1)$}  \label{it1}\\ 
 & \lim_{T \to \infty}\frac{\mathbb{E} F_T^2}{T \log T}
  =   \frac{9}{8}\lim_{T \to \infty}\frac{J_T}{\log T}    
&\qquad\hbox{when $H=\frac 34$},  \label{it2}  
\end{empheq} 
where 
\[
J_T = \int_{[0,T]^3}e^{-\theta |T-u_2|- \theta|s_1-u_1|}(T-s_1)^{2H-2}|u_2-u_1|^{2H-2}du_1du_2ds_1\,.
\]
Denote 
\[
  h(T) = T^{3-4H} {\bf{1}}_{\{H \in (\frac{3}{4}, 1)\}} + (\log T)^{-1} {\bf{1}}_{\{H=\frac{3}{4}\}} \,.
\]
Then, finding the  limits \eqref{it1} and \eqref{it2} is reduced to the computation of $\lim_{T \to \infty}h(T)J_T$. \\

Making the change of variables $x=T-u_2$, $y=u_1-s_1$ and $ z=T-s_1$ in the  region $\{u_1 > s_1\}$
and the change of variables $x=T-u_2$,  $ y=s_1-u_1$, $ z=T-s_1$ in the region $\{u_1 < s_1\}$, we can write $J_T$ as follows
\begin{eqnarray}
   J_T & = &  \int_{[0,T]^3, y<z} e^{-\theta (x+y)} z^{2H-2} |x+y-z|^{2H-2} dxdydz \nonumber \\
   & & + \, \int_{[0,T]^3, y+z < T} e^{-\theta (x+y)} z^{2H-2} |y+z-x|^{2H-2} dxdydz \label{jt.s}\,.
\end{eqnarray}
Consider the functions
\[
  f_1(x,y,z) = e^{-\theta (x+y)} z^{2H-2} |x+y-z|^{2H-2} \,, \qquad f_2(x,y,z)= e^{-\theta (x+y)} z^{2H-2} |y+z-x|^{2H-2} \,.
\]
For the first integral of \eqref{jt.s}, we split the integration interval $\{y<z \}$ into $\{x+y < z\} \cup \{x+y \geq z, y<z\}$. For the second integral of \eqref{jt.s}, we write the integration interval as $\{y+z < T\} = \{ x+y < T, x \leq y \} \cup \{ x+y < T, 0 < x- y < z \} \cup \{ x+y < T, x-y \geq z \} \cup \{x+y \geq T\} \backslash \{y+z \geq T\}$. In this way, we can split $J_T$ into seven integrals. It turns out  that some of them are bounded by a constant independent of $T$ and they do not contribute to the limit, because  $h(T) \to 0$. More precisely, we  can derive the following bounds:
\begin{eqnarray*}
 \int_{[0,T]^3, x+y \geq z, y<z} f_1(x,y,z)dxdydz & \leq & \int_{[0,T]^3, x+y \geq z} f_1(x,y,z)dxdydz \\
 & = & C \int_{[0,T]^2} e^{-\theta(x+y)} (x+y)^{4H-3} dxdy \le  M \,,
\end{eqnarray*}
where  in the second step we  integrated in  $z$ and the last step follows  from the inequality $x+y \geq 2\sqrt{xy}$. It is trivial to show that
\[
  \int_{[0,T]^3, x+y \geq T} f_2(x,y,z) dxdydz \leq e^{-\theta T} \int_{[0,T]^3}  z^{2H-2} |y+z-x|^{2H-2} dxdydz \le  M \,,
\]
and 
\begin{eqnarray*}
  \int_{[0,T]^3, x+y < T, x-y \geq z} f_2(x,y,z) dxdydz & \leq & \int_{[0,T]^3,  x-y \geq z} e^{-\theta(x+y)} z^{2H-2}(x-y-z)^{2H-2} dxdydz \\
 & = & C \int_{[0,T]^2} e^{-\theta (x+y)} (x-y)^{4H-3} dxdy \le  M \,.
\end{eqnarray*}
The last bounded integral is 
\begin{eqnarray*}
 \int_{[0,T]^3, y+z \geq T} f_2(x,y,z) dxdydz &\leq& \int_{[0,T]^3, y+z \geq T} e^{-\theta (x+y)} z^{2H-2} (T-x)^{2H-2} dxdydz \\
 & \leq &  \left( \int_0^T e^{-\theta x} (T-x)^{2H-2} dx \right)^2 \le  M \,,
\end{eqnarray*}
where in the second step we have used the inequality $z^{2H-2} \leq (T-y)^{2H-2}$ and the last step follows from the following inequality
\[
\int_0^T e^{-\theta x} (T-x)^{2H-2} dx \leq \int_0^{T/2} e^{-\theta x} x^{2H-2}dx + \int_{T/2}^T  e^{-\theta (T-x)} (T-x)^{2H-2} dx \leq 2\int_0^{\infty} e^{-\theta x} x^{2H-2} dx \,.
\]
With these observations, 
\begin{eqnarray*}
\lim_{T\to \infty}h(T)J_T & = & \lim_{T \to \infty} h(T) \int_{x+y<z}f_1(x,y,z) dxdydz + \lim_{T \to \infty} h(T) \int_{\{ x+y < T, x \leq y \}}f_2(x,y,z) dxdydz  \\
& & + \lim_{T \to \infty} h(T) \int_{ \{ x+y < T, 0 < x- y < z \}} f_2(x,y,z) dxdydz \,.
\end{eqnarray*}
We make change of variables $z-(x+y) \to u, x+y \to v, y \to y$ for the first term, $y-x \to u, z \to v, y \to y$ for the second term, and $x-y \to u, z-x+y \to v, y \to y$ for the third term. In this way, we obtain
\begin{eqnarray*}
\lim_{T\to \infty}h(T)J_T & = & \lim_{T\to \infty} h(T) \int_{[0,T]^3, u+v<T, y<v} e^{-\theta v} (u+v)^{2H-2} u^{2H-2} dydudv \\
& & + \lim_{T\to \infty} h(T) \int_{[0,T]^3, u < y < (T+u)/2} e^{-\theta(-u+2y)} v^{2H-2} (u+v)^{2H-2} dydudv  \\
& & + \lim_{T\to \infty} h(T) \int_{[0,T]^3, u+v<T, y<(T-u)/2} e^{-\theta(u+2y)} (u+v)^{2H-2} v^{2H-2} dy dudv \,.
\end{eqnarray*}
Finally, the limits (\ref{e.lemm5.1-2}) and (\ref{e.lemm5.1-3}) follow from integrating in the variable $y$ and an application of  Lemma \ref{at}. \\

We proceed now to the proof of (\ref{e.lemm5.1-4}). Assume $H > 3/4$. Recall that $R_1 = I_2(\delta_{0,1})$ is given in Theorem \ref{thetaclt} and $\widetilde F_T$ is given by \eqref{e.ft2}. By \eqref{ip.rosenb}, we can write 
\begin{equation*}
\mathbb{E}( R_1 (T^{1-2H}\widetilde F_T)) = 2\alpha_H^2 T \int_{[0,1]^3} e^{-\theta T|t-s|} |t-t'|^{2H-2} |s-t'|^{2H-2}dsdtdt'\,.
\end{equation*}
We make the change of variables $Tt \to x, Ts \to y, Tt' \to z$ to rewrite the above equation as
\[
  \mathbb{E}(T^{1-2H} R_1 \widetilde F_T) = \frac{2\alpha_H^2}{T^{4H-2}} \int_{[0,T]^3} e^{-\theta |x-y|} |x-z|^{2H-2} |y-z|^{2H-2} dxdydz \,.
\]
By the symmetry of $x, y$ in the above equation, applying L'Hopital's rule  yields
\begin{eqnarray}
  \lim_{T \to \infty}\mathbb{E}(T^{1-2H} R_1 \widetilde F_T) & = & \frac{\alpha_H^2}{2H-1} \lim_{T \to \infty} T^{3-4H} \Big(2 \int_{[0,T]^2} e^{-\theta (T-y)} (T-z)^{2H-2} |y-z|^{2H-2} dydz \nonumber \\
  & & + \int_{[0,T]^2} e^{-\theta|x-y|} (T-x)^{2H-2} (T-y)^{2H-2} dxdy\Big) \\
  & =: & \frac{\alpha_H^2}{2H-1} \lim_{T \to \infty} T^{3-4H} (2 L_1 + L_2) \,.
\end{eqnarray}
To compute $L_1$, on the region $\{y > z\}$ we make the change of variables $y-z \to t, \; T - y \to s$  and on the region $\{ y < z \}$ we make the change of variables $z-y \to s, \; T-z \to t$. In this way we obtain
\begin{equation*}
  L_1 = \int_{[0,T]^2, s+t<T} e^{-\theta s} (s+t)^{2H-2} t^{2H-2} dsdt + \int_{[0,T]^2, s+t<T} e^{-\theta (s+t)} t^{2H-2} s^{2H-2} dsdt
\end{equation*}
For the term $L_2$, by  symmetry it is sufficient to consider the region $\{x>y\}$ and  making the change of variables $T - x \to t, \; x-y \to s$, we obtain
\[
  L_2 = 2 \int_{[0,T]^2, s+t <T} e^{-\theta s} t^{2H-2} (s+t)^{2H-2} dsdt \,.
\]
Notice that the second summand of $L_1$ is bounded by $\int_{[0,\infty)^2} e^{-\theta (s+t)} t^{2H-2} s^{2H-2} dsdt$. Therefore,
\begin{eqnarray*}
   \lim_{T \to \infty}\mathbb{E}(T^{1-2H} R_1 \widetilde F_T) & = & \frac{4\alpha_H^2}{2H-1} \lim_{T \to \infty} T^{3-4H} \int_{[0,T]^2, s+t<T} e^{-\theta s} (s+t)^{2H-2} t^{2H-2} dsdt \\
   & = & \frac{4\alpha_H^2 \theta^{-1}}{(2H-1)(4H-3)} \,,
\end{eqnarray*}
where the last step is due to Lemma \ref{at}.  This finishes the proof of Lemma \ref{lem5.1}.
\end{proof}

\begin{lem} \label{yt.lim}
 Let   $Y_T$ be defined by
   \begin{equation}
	   Y_t = \sigma \int^t_{-\infty} e^{-\theta (t-s)} dB_s^H = X_t + e^{-\theta t} \xi \, , \label{yt} 
   \end{equation}
   where
   \begin{equation}
	   \xi = \sigma \int^0_{-\infty}e^{\theta s}dB_s^H \, . \label{xi} 
   \end{equation}
     For any $\alpha > 0$,  $\dfrac{Y_T}{T^{\alpha}}$  
 converges almost surely to zero as $T$ tends to infinity.	
\end{lem}

\begin{proof}
 The case  $H \geq \frac{1}{2}$ was  proved  in \cite{hn}.  Here,  we present a different  proof  valid for all
 $H \in (0,1)$. 
%
    We denote $\beta := \mathbb{E} \xi^2 = \sigma^2 \theta^{-2H} H\Gamma(2H)$, which is computed in Lemma \ref{xt2}. Notice that the covariance of the process $Y_t$ for $t>0$ is computed as
  \begin{eqnarray*}
  	\cov(Y_0, Y_t) & = & e^{-\theta t} \mathbb{E}\Big(\xi \Big[\xi + \sigma \int^t_0 e^{\theta u} dB^H_u\Big]\Big) \\
	& = & e^{-\theta t} \beta + e^{-\theta t} \sigma^2 \mathbb{E}\Big(\int^0_{-\infty} e^{\theta s} dB_s^H \int^t_0 e^{\theta u}dB_u^H\Big) \,. 
  \end{eqnarray*}
 We use integration by parts for both integrals in the above equation to rewrite
    \[
    	\cov(Y_0, Y_t) = e^{-\theta t} \beta + g_1(t) - g_2(t) \, ,
    \]
 where 
    \[
	g_1(t) = e^{-\theta t} \sigma^2 \theta^2 \mathbb{E}\Big(\int^0_{-\infty} \int^t_0 B_s^H B_u^H e^{\theta (u+s)}duds \Big) \,, \quad
	g_2(t) = \sigma^2 \theta \mathbb{E}\Big(\int^0_{-\infty} B_s^H B_t^H e^{\theta s} ds\Big) \,.
	\]
 By Fubini theorem and the explicit form of the covariance of fBm,
  \begin{eqnarray*}
	g_1(t) & = & \frac{1}{2} e^{-\theta t} \sigma^2 \theta^2 \int^0_{-\infty} \int^t_0 (|s|^{2H}+ u^{2H}-(u-s)^{2H}) e^{\theta (u+s)}duds \\
	& = & \beta(1-e^{-\theta t}) + \frac{1}{2}e^{-\theta t} \sigma^2 \theta \int_0^t e^{\theta u}u^{2H} du - \frac{\beta}{2} (e^{\theta t} - e^{-\theta t}) \,.
  \end{eqnarray*}
 When we compute the above double integral, we write the integrand as three items by distributing $e^{\theta (u+s)}$ and then integrate the terms one by one. For the term involving $(u-s)^{2H}$, we make the change of variables $u-s \to x, s \to y$ and integrate in the variable $y$ first.  Similarly,
  \begin{eqnarray*}
    g_2(t) & = & \frac{1}{2} \sigma^2 \theta \int^0_{-\infty}  (|s|^{2H}+ t^{2H}-(t-s)^{2H}) e^{\theta s} ds \\
	       & = & \beta + \frac{1}{2} \sigma^2 t^{2H} - \beta e^{\theta t} + \frac{1}{2}\sigma^2 \theta e^{\theta t} \int_0^t e^{-\theta s} s^{2H} ds \,.
  \end{eqnarray*}
Denote $a_t = o(b_t)$ if $\lim_{t \to 0} \frac{a_t}{b_t} = 0$. Notice that $\int_0^t e^{\theta(u-t)}u^{2H} du - \int_0^t e^{\theta(t-s)}s^{2H} ds = o(t^{2H})$. Based on the above computations, for $t$ small, we have
 \[
  \cov(Y_0, Y_t) = \beta \Big[1 - \frac{\theta^{2H}}{\Gamma(2H+1)}t^{2H}+o(t^{2H})\Big] \,.
 \]  
The lemma now follows  from Theorem 3.1 of Pickands \cite{pi}. \\
\end{proof}

\begin{lem} \label{xt2}  Let the stochastic process $X_t$ satisfy \eqref{e.ou-time-varying} (with $\si_t=\si$).  Then  
  $\displaystyle\frac{1}{T}\int^T_0 X_t^2 dt \rightarrow \sigma^2\theta^{-2H}H\Gamma(2H)$ a.s. and in $L^2$, as $T \rightarrow \infty$. 	
\end{lem}

\begin{proof} When  $H\geq \frac{1}{2}$,  the Lemma is  proved in \cite{hn}. 
  We shall handle the 
 case of general Hurst parameter in a similar way. 
The   process $\{Y_t, t\geq 0\}$  defined by \eqref{yt} 
 is Gaussian, stationary and ergodic for all $H \in (0,1)$.  By the ergodic theorem, 
 \[
 \frac{1}{T} \int^T_0 Y_t^2 dt \rightarrow \mathbb{E}(Y_0^2),\quad \hbox{as $T$ goes to infinity,} 
 \]
almost surely and in $L^2$.  This implies $$\frac{1}{T} \int^T_0 X_t^2 dt \rightarrow \mathbb{E}(Y_0^2),$$ as $T$ goes to infinity, almost surely and in $L^2$.  Moreover, integrating by parts yields 
\begin{eqnarray*}
\mathbb{E}(Y_0^2) =
\mathbb{E}(\xi^2) &=&  \sigma^2 \mathbb{E}\Big(\int^0_{-\infty}e^{\theta s}dB_s^H\Big)^2 = \theta^2 \sigma^2 \mathbb{E} \int^0_{-\infty} \int^0_{-\infty} B_s^H B_r^H e^{\theta (s+r)}dsdr\\
&=& \theta^2 \sigma^2 \int^{\infty}_0 \int^{\infty}_0 e^{-\theta (s+r)} R_H(s,r) dsdr = \sigma^2 \theta^{-2H} H \Gamma(2H)\,.
\end{eqnarray*}
In the last step of the above computation, we use the same idea as near the end of the proof for Lemma \ref{yt.lim}.  Namely,
one writes out the explicit form of $R_H(s,r)$, split the integrand into three items by distributing $e^{-\theta (s+r)}$ to the summands of $R_H(s,r)$, and then integrate the three items one by one.  For the item involving $|s-r|^{2H}$, noticing the symmetry of $s,r$, one can make change of variables $s-r \to u, r \to v$, and then integrate in the variable $v$ first.
\end{proof}






\begin{minipage}{1.0\textwidth}
\vskip 1cm
Yaozhong Hu, David Nualart and Hongjuan Zhou: Department of Mathematics, University of Kansas, 405 Snow Hall, Lawrence, Kansas, 66045, USA. 

{\it E-mail address:} yhu@ku.edu,    nualart@ku.edu,  zhj@ku.edu
\end{minipage}

\end{document}